%% file: ARPDC_final.tex
\documentclass[11pt]{article}

\usepackage{epsfig,epsf,fancybox}
\usepackage{amsmath}
\usepackage{mathrsfs}
\usepackage{amssymb}
\usepackage{graphicx}
\usepackage{color}
\usepackage[linktocpage,pagebackref,colorlinks,linkcolor=blue,anchorcolor=blue,citecolor=blue,urlcolor=blue,hypertexnames=false]{hyperref}
\usepackage{boxedminipage}
\usepackage{stmaryrd}
\usepackage{multirow}
\usepackage{booktabs}
\usepackage{accents}
\usepackage{cite}
\usepackage{float}
\usepackage[ruled,vlined,linesnumbered]{algorithm2e}

\newtheorem{theorem}{Theorem}[section]

\newtheorem{lemma}[theorem]{Lemma}

\newtheorem{definition}{Definition}[section]

\newcommand{\diag}{\textnormal{diag}}

\,
\,

\newcommand{\be}{\begin{equation}}
\newcommand{\ee}{\end{equation}}
\newcommand{\ba}{\begin{array}}
\newcommand{\ea}{\end{array}}
\newcommand{\bpm}{\begin{pmatrix}}
\newcommand{\epm}{\end{pmatrix}}

\newcommand{\bea}{\begin{eqnarray}}
\newcommand{\eea}{\end{eqnarray}}
\newcommand{\beaa}{\begin{eqnarray*}}
\newcommand{\eeaa}{\end{eqnarray*}}

\newcommand{\bal}{\begin{align}}
\newcommand{\eal}{\end{align}}
\newcommand{\baln}{\begin{align*}}
\newcommand{\ealn}{\end{align*}}

\input{macros.tex}
\begin{document}

\title{Asynchronous parallel primal-dual block coordinate update methods for affinely constrained convex programs\thanks{This work is partly supported by NSF grant DMS-1719549.}}

\author{Yangyang Xu\thanks{\url{xuy21@rpi.edu}. Department of Mathematical Sciences, Rensselaer Polytechnic Institute, Troy, NY}
}

\date{}

\maketitle

\begin{abstract}

Recent several years have witnessed the surge of asynchronous (async-) parallel computing methods due to the extremely big data involved in many modern applications and also the advancement of multi-core machines and computer clusters. In optimization, most works about async-parallel methods are on unconstrained problems or those with block separable constraints.

In this paper, we propose an async-parallel method based on block coordinate update (BCU) for solving convex problems with \emph{nonseparable} linear constraint. Running on a single node, the method becomes a novel randomized primal-dual BCU for multi-block affinely constrained problems. For these problems, Gauss-Seidel cyclic primal-dual BCU is not guaranteed to converge to an optimal solution if no additional assumptions, such as strong convexity, are made. On the contrary, assuming convexity and existence of a primal-dual solution, we show that the objective value sequence generated by the proposed algorithm converges \emph{in probability} to the optimal value and also the constraint residual to zero. In addition, we establish an ergodic $O(1/k)$ convergence result, where $k$ is the number of iterations. Numerical experiments are performed to demonstrate the efficiency of the proposed method and significantly better speed-up performance than its sync-parallel counterpart.

\vspace{0.3cm}

\vspace{0.3cm}


\end{abstract}

\section{Introduction}
Modern applications in various data sciences and engineering can involve huge amount of data and/or variables \cite{WH2014big}. Driven by these very large-scale problems and also the advancement of multi-core computers, parallel computing has gained tremendous attention in recent years. 
In this paper, we consider the affinely constrained multi-block structured problem:
\begin{equation}\label{eq:mb-prob}
\min_\vx f(\vx_1,\ldots,\vx_m)+\sum_{i=1}^m g_i(\vx_i), \st \sum_{i=1}^m \vA_i\vx_i=\vb,
\end{equation}
where the variable $\vx$ is partitioned into multiple disjoint blocks $\vx_1,\ldots,\vx_m$, $f$ is a continuously differentiable and convex function, and each $g_i$ is a lower semi-continuous extended-valued convex but possibly non-differentiable function. Besides the nonseparable affine constraint, \eqref{eq:mb-prob} can also include certain block separable constraint by letting part of $g_i$ be an indicator function of a convex set, e.g., nonnegativity constraint. 

We will present a novel asynchronous (async-) parallel primal-dual method (see Algorithm \ref{alg:async-rpdc}) towards finding a solution to \eqref{eq:mb-prob}. Suppose there are multiple nodes (or cores, CPUs). We let one node (called \emph{master node}) update both primal and dual variables and all the remaining ones (called \emph{worker nodes}) compute and provide block gradients of $f$ to the master node. We assume each $g_i$ is proximable (see the definition in \eqref{eq:prox} below). When there is a single node, our method reduces to a novel serial primal-dual BCU for solving \eqref{eq:mb-prob}; see Algorithm \ref{alg:rpdc}. 

\subsection{Motivating examples}
Problems in the form of \eqref{eq:mb-prob} arise in many areas including signal processing, machine learning, finance, and statistics. For example, the basis pursuit problem \cite{chen2001atomic} seeks a sparse solution on an affine subspace through solving the linearly constrained program:
\begin{equation}\label{eq:bp}
\min_{\vx} \|\vx\|_1, \st \vA\vx = \vb.
\end{equation}
Partitioning $\vx$ into multiple disjoint blocks in an arbitrary way, one can formulate \eqref{eq:bp} into the form of \eqref{eq:mb-prob} with $f(\vx)=0$ and each $g_i(\vx_i)=\|\vx_i\|_1$. 

Another example is the portfolio optimization \cite{markowitz1952portfolio}. Suppose we have a unit of capital to invest on $m$ assets. Let $x_i$ be the fraction of capital invested on the $i$-th asset and $\xi_i$ be the expected return rate of the $i$-th asset. The goal is to minimize the risk measured by $\sqrt{\vx^\top \bm{\Sigma} \vx}$ subject to total unit capital and minimum expected return $c$, where $\vx=(x_1;\ldots;x_m)$ and $\bm{\Sigma}$ is the covariance matrix. To find the optimal $\vx$, one can solve the problem:
\begin{equation}\label{eq:portfo}
\min_\vx \frac{1}{2}\vx^\top \bm{\Sigma} \vx, \st \sum_{i=1}^m x_i \le 1,\, \sum_{i=1}^m\xi_i x_i \ge c,\, x_i\ge 0,\forall i.
\end{equation} 
Introducing slack variables to the first two inequalities, one can easily write \eqref{eq:portfo} into the form of \eqref{eq:mb-prob} with a quadratic $f$ and each $g_i$ being an indicator function of the nonnegativity constraint set. 

In addition, \eqref{eq:mb-prob} includes as a special case the dual support vector machine (SVM) \cite{cortes1995support}. Given training data set $\{(\vx_i, y_i)\}_{i=1}^N$ with $y_i\in\{-1, +1\},\,\forall i$, let $\vX=[\vx_1,\ldots,\vx_N]$ and $\vy=[y_1;\ldots;y_N]$. The dual form of the linear SVM can be written as
\begin{equation}\label{eq:dsvm}
\min_\vtheta \frac{1}{2}\vtheta^\top \Diag(\vy)\vX^\top\vX\Diag(\vy)\vtheta - \ve^\top \vtheta, \st \vy^\top\vtheta=0,\, 0\le \theta_i\le C,\,\forall i,
\end{equation}
where $\vtheta=[\theta_1;\ldots;\theta_N]$, and $C$ is a given number relating to the soft margin size. It is easy to formulate \eqref{eq:dsvm} into the form of \eqref{eq:mb-prob} with $f$ being the quadratic objective function and each $g_i$ the indicator function of the set $[0,C]$. 

Finally, the penalized and constrained (PAC) regression problem \cite{james2013pcreg} is also one example of \eqref{eq:mb-prob} with $f(\vx)=\frac{1}{N} \sum_{j = 1}^N f_j(\vx)$ and linear constraint of $J$ equations. As $N\gg J$ (that often holds for problems with massive training data), the PAC regression satisfies Assumption \ref{assump:expensive-fi}. In addition, if $m\gg 1$ and $N\gg 1$, both \eqref{eq:portfo} and \eqref{eq:dsvm} satisfy the assumption, and thus the proposed async-parallel method will be efficient when applied to these problems. Although Assumption \ref{assump:expensive-fi} does not hold for \eqref{eq:bp} as $p>1$, our method running on a single node can still outperform state-of-the-art non-parallel solvers; see the numerical results in section \ref{sec:bp}.

\subsection{Block coordinate update}
The block coordinate update (BCU) method breaks possibly very high-dimensional variable into small pieces and renews one at a time while all the remaining blocks are fixed. Although the problem \eqref{eq:mb-prob} can be extremely large-scale and complicated, BCU solves a sequence of small-sized and easier subproblems. As \eqref{eq:mb-prob} owns nice structures, e.g., coordinate friendly \cite{peng2016cf}, BCU can not only have low per-update complexity but also enjoy faster overall convergence than the method that updates the whole variable every time. BCU has been applied to many unconstrained or block-separably constrained optimization problems (e.g., \cite{Tseng-01, tseng2009_CGD, nesterov2012cd, razaviyayn2013unified, XY_2013_multiblock, richtarik2014iteration, XY_2017_ecd, hong2017iteration-bcd}), and it has also been used to solve affinely constrained separable problems, i.e., in the form of \eqref{eq:mb-prob} without $f$ term (e.g., \cite{feng2015cvg-LADMM, deng2017parallel, He-Hou-Yuan, he2012alternating, he2017convergence}). However, only a few existing works (e.g., \cite{hong2014block, gao2017first, GXZ-RPDCU2016}) have studied BCU on solving affinely constrained problems with a nonseparable objective function.  

\subsection{Asynchronization}
Parallel computing methods distribute computation over and collect results from multiple nodes. Synchronous (sync) parallel methods require all nodes to keep in the same pace. Upon all nodes finish their own computation, they altogether proceed to the next step. This way, the faster node has to wait for the slowest one, and that wastes a lot of waiting time. On the contrary, async-parallel methods keep all nodes continuously working and eliminate the idle waiting time. Numerous works (e.g., \cite{recht2011hogwild, liu2015async-scd, liu2015asynchronous-JMLR, Peng_2016_AROCK}) have demonstrated that async-parallel methods can achieve significantly better speed-up than their sync-parallel counterparts.

Due to lack of synchronization, the information used by a certain node may be outdated. Hence the convergence of an async-parallel method cannot be easily inherited from its non-parallel counterpart but often requires a new tool of analysis. Most existing works only analyze such methods for unconstrained or block-separably constrained problems. Exceptions include \cite{wei20131, zhang2014asynchronous, chang2016asynchronous-I, chang2016asynchronous-II} that consider separable problems with special affine constraint.

\subsection{Related works}
Recent several years have witnessed the surge of async-parallel methods partly due to the increasingly large scale of data/variable involved in modern applications. However, only a few existing works discuss such methods for affinely constrained problems. Below we review the literature of async-parallel BCU methods in optimization and also primal-dual BCU methods for affinely constrained problems. 

It appears that the first async-parallel method was proposed by Chazan and Miranker \cite{D-W1969chaotic-relax} for solving linear systems. Later, such methods have been applied in many others fields. In optimization, the first async-parallel BCU method was due to Bertsekas and Tsitsiklis \cite{bertsekas1989parallel} for problems with a smooth objective. It was shown that the objective gradient sequence converges to zero. Tseng \cite{tseng1991rate} further analyzed its convergence rate and established local linear convergence by assuming isocost surface separation and a local Lipschitz error bound on the objective. Recently, \cite{liu2015asynchronous-JMLR, liu2015async-scd} developed async-parallel methods based on randomized BCU for convex problems with possibly block separable constraints. They established convergence and also rate results by assuming a bounded delay on the outdated block gradient information. The results have been extended to the case with unbounded probabilistic delay in \cite{peng2016-total-async}, which also shows convergence of the async-parallel BCU methods for nonconvex problems. On solving problems with convex separable objective and linear constraints, \cite{wei20131} proposed to apply the alternating direction method of multipliers (ADMM) in an asynchronous and distributive way. Assuming a special structure on the linear constraint, it established $O(1/k)$ ergodic convergence result, where $k$ is the total number of iterations. In \cite{zhang2014asynchronous, chang2016asynchronous-I, chang2016asynchronous-II, bianchi2016coordinate}, the async-ADMM is applied to distributed multi-agent optimization, which can be equivalently formulated into \eqref{eq:mb-prob} with $f=0$ and consensus constraint. Among them, \cite{bianchi2016coordinate} proved an almost sure convergence result, \cite{zhang2014asynchronous} showed sublinear convergence of the async-ADMM for convex problems, and \cite{chang2016asynchronous-II} established its linear convergence for strongly convex problems.  Besides convex problems, \cite{chang2016asynchronous-I} also considered nonconvex cases. Assuming certain structure on the problem and choosing appropriate parameters, it showed that any limit point of the iterates satisfies first-order optimality conditions. 
The works \cite{Peng_2016_AROCK, combettes2016asynchronous} developed async-parallel BCU methods for fixed-point or monotone inclusion problems. Although these settings are more general (including convex optimization as a special case), no convergence rate results have been shown under monotonicity assumption\footnote{In \cite{Peng_2016_AROCK}, a linear convergence result is established under strong monotonicity assumption, which is similar to strong convexity in optimization.} (similar to convexity in optimization).

Running on a single node, the proposed async-parallel method reduces to a serial randomized primal-dual BCU. In the literature, various Gauss-Seidel (GS) cyclic BCU methods have been developed for solving separable convex programs with linear constraints. Although a cyclic primal-dual BCU can empirically work well, in general it may diverge \cite{feng2015cvg-LADMM, chen2016direct, xu2018hybrid-JGS}. By an example of $3\times 3$ linear system, \cite{chen2016direct} showed that the direct extension of ADMM could diverge on solving problems with more than 2 blocks. The works \cite{feng2015cvg-LADMM, xu2018hybrid-JGS} showed that even with proximal terms, the cyclic primal-dual BCU can still diverge. Hence, to guarantee convergence, additional assumptions besides convexity must be made, such as strong convexity on part of the objective \cite{han2012note, cai2017convergence, li2015convergent, LMZ2015JORSC, lin2015global, davis2017three-op} and orthogonality properties of block matrices in the linear constraint \cite{chen2016direct}. Assuming strong convexity of each block component function and choosing the penalty parameter within a region, \cite{han2012note} showed the convergence of ADMM to an optimal solution for solving problems with multiple blocks. For 3-block problems, \cite{cai2017convergence, li2015convergent, davis2017three-op} established the convergence of ADMM and/or its variant by assuming strong convexity of one block component function. For general $m$-block problems, \cite{LMZ2015JORSC} showed that if $m-1$ block component functions are strongly convex, then ADMM with appropriate penalty parameter is guaranteed to converge. Without these assumptions, modifications to the algorithm are necessary for convergence. For example, \cite{he2012alternating, he2017convergence} performed a correction step after each cycle of updates. On solving linear system or quadratic programming,  \cite{sun2015expected} proposed, at each iteration, to first randomly permute all block variables and then perform a cyclic update. Jacobi-type update together with proximal terms was used in \cite{deng2017parallel, He-Hou-Yuan} to ensure the convergence of the algorithm, which turns out to be a linearized augmented Lagrange method (ALM). In addition, a hybrid Jacobi-GS update was performed in \cite{sun2015convergent, li2016schur, xu2018hybrid-JGS}. Different from these modifications, our algorithm simply employs randomization in selecting block variable and can perform significantly better than Jacobi-type methods. In addition, convergence is guaranteed with convexity assumption and thus better than those results for GS-type methods.

\subsection{Contributions} The contributions are summarized as follows.
\begin{itemize}
\renewcommand\labelitemi{--}
\item We propose an async-parallel BCU method for solving multi-block structured convex programs with linear constraint. The algorithm is the first async-parallel primal-dual method for affinely constrained problems with \emph{nonseparable} objective. When there is only one node, it reduces to a novel serial primal-dual BCU method.
\item With convexity and existence of a primal-dual solution, convergence of the proposed method is guaranteed. We first establish convergence of the serial BCU method. We show that the objective value converges in probability to the optimal value and also the constraint residual to zero. In addition, we establish an ergodic convergence rate result. Then through bounding a cross term involving delayed block gradient, we prove that similar convergence results hold for the async-parallel BCU method if a delay-dependent stepsize is chosen.
\item We implement the proposed algorithm and apply it to the basis pursuit, quadratic programming, and also the support vector machine problems. Numerical results demonstrate that the serial BCU is comparable to or better than state-of-the-art methods. In addition, the async-parallel BCU method can achieve significantly better speed-up performance than its sync-parallel counterpart. 
\end{itemize}

\subsection{Notation and Outline}\label{sec:notation} 
We use bold small letters $\vx, \vy, \vlam, \ldots$ for vectors and bold capital letters $\vA, \vL, \vP, \ldots$ for matrices. $[m]$ denotes the integer set $\{1,2,\ldots,m\}$. $\vU_i\vx$ represents a vector with $\vx_i$ for its $i$-th block and zero for all other $m-1$ blocks. $\mathrm{blkdiag}(\vP_1,\ldots,\vP_m)$ denotes a block diagonal matrix with $\vP_1,\ldots,\vP_m$ on the diagonal blocks. We denote $\|\vx\|$ as the Euclidean norm of $\vx$ and $\|\vx\|_\vP=\sqrt{\vx^\top \vP \vx}$ for a symmetric positive semidefinite matrix $\vP$. We reserve $\vI$ for the identity matrix, and its size is clear from the context. $\EE_{i_k}$ stands for the expectation about $i_k$ conditional on previous history $\{i_1,\ldots,i_{k-1}\}$. We use $\bm{\xi}^k\overset{p}\to\bm{\xi}$ for convergence in probability of a random vector sequence $\bm{\xi}^k$ to $\bm{\xi}$.    

For ease of notation, we let $g(\vx)=\sum_{i=1}^m g_i(\vx_i)$, $F=f+g$, and $\vA=[\vA_1,\ldots,\vA_m]$. Denote $$\Phi(\bar{\vx},\vx,\vlam)=F(\bar{\vx})-F(\vx)-\langle\vlam, \vA\bar{\vx}-\vb\rangle.$$ Then $(\vx^*,\vlam^*)$ is a saddle point of \eqref{eq:mb-prob} if $\vA\vx^*=\vb$ and $\Phi(\vx,\vx^*,\vlam^*)\ge 0,\,\forall \vx$.

The proximal operator of a function $\psi$ is defined as
\begin{equation}\label{eq:prox}\prox_\psi(\vx) = \argmin_\vy \psi(\vy) + \frac{1}{2}\|\vx-\vy\|^2.
\end{equation}
If $\prox_\psi(\vx)$ has a closed-form solution or is easy to compute, we call $\psi$ \emph{proximable}.

\textbf{Outline.} The rest of the paper is organized as follows. In section \ref{sec:algorithm}, we present the serial and also async-parallel primal-dual BCU methods for \eqref{eq:mb-prob}. Convergence results of the algorithms are shown in section \ref{sec:analysis}. Section \ref{sec:numerical} gives experimental results, and finally section \ref{sec:conclusion} concludes the paper.

\section{Algorithm}\label{sec:algorithm}
In this section, we propose an async-parallel primal-dual method for solving \eqref{eq:mb-prob}. Our algorithm is a BCU-type method based on the augmented Lagrangian function of \eqref{eq:mb-prob}:
$$\cL_\beta(\vx,\vlam) = f(\vx) + g(\vx) - \langle \vlam, \vA \vx -\vb\rangle + \frac{\beta}{2}\|\vA\vx-\vb\|^2,$$
where $\vlam$ is the multiplier (or augmented Lagrangian dual variable), and $\beta$ is a penalty parameter. 

\subsection{Non-parallel method}
For ease of understanding, we first present a non-parallel method in Algorithm \ref{alg:rpdc}. At every iteration, the algorithm chooses one out of $m$ block uniformly at random and renews it by \eqref{eq:update-x} while fixing all the remaining blocks. Upon finishing the update to $\vx$, it immediately changes the multiplier $\vlam$. The linearization to possibly complicated smooth term $f$ greatly eases the $\vx$-subproblem. Depending on the form of $g_i$, we can choose appropriate $\vP_i$ to make \eqref{eq:update-x} simple to solve. Since each $g_i$ is proximable, one can always easily find a solution to  \eqref{eq:update-x} if $\vP_i=\eta_i \vI$. For even simpler $g_i$ such as $\ell_1$-norm and indicator function of a box constraint set, we can set $\vP_i$ to a diagonal matrix and have a closed-form solution to  \eqref{eq:update-x}. Note that the algorithm is a special case of Algorithm 1 in \cite{GXZ-RPDCU2016} with only one group of variables. We include it here for ease of understanding our parallel method.

Randomly choosing a block to update has advantages over the cyclic way in both theoretical and empirical perspectives. We will show that \emph{this randomized BCU has guaranteed convergence with convexity other than strong convexity assumed by the cyclic primal-dual BCU}. In addition, randomization enables us to parallelize the algorithm in an efficient way as shown in Algorithm \ref{alg:async-rpdc}.  

\begin{algorithm}[h]\caption{Randomized primal-dual block update for \eqref{eq:mb-prob}}\label{alg:rpdc}
\DontPrintSemicolon
\textbf{Initialization:} choose $\vx^0$ and $\vlam^0=\vzero$; let $\vr^0=\vA\vx^0-\vb$ and $k=0$; set $\beta,\rho$ and $\vP_i$'s.\;
\While{the stopping conditions not satisfied}{
Pick $i_k$ from $[m]$ uniformly at random.\;
For any $i\neq i_k$, keep $\vx_i^{k+1}=\vx^k_i$, and for $i=i_k$, update $\vx_i$ by
\begin{equation}
\vx_i^{k+1}\in\argmin_{\vx_i}\langle\nabla_i f(\vx^k)-\vA_i^\top(\vlam^k-\beta \vr^k), \vx_i\rangle +g_i(\vx_i)+\frac{1}{2}\|\vx_i-\vx_i^k\|_{\vP_i}^2,\label{eq:update-x}
\end{equation}
Update residual $\vr$ and multipliers $\vlam$ by
\begin{align}
&\vr^{k+1}=\vr^k+\vA_{i_k}(\vx_{i_k}^{k+1}-\vx_{i_k}^k),\label{eq:update-r}\\
&\vlam^{k+1}=\vlam^k-\rho \vr^{k+1}.\label{eq:update-lam}
\end{align}
Let $k\gets k+1$.
}
\end{algorithm}

\subsection{Async-parallel method}

Assume there are $p$ nodes. Let the data and variables be stored in a global memory accessible to every node. We let one node (called \emph{master node}) update both primal variable $\vx$ and dual variable $\vlam$ and the remaining ones (called \emph{worker node}s) compute block gradients of $f$ and provide them to the master node. The method is summarized in Algorithm \ref{alg:async-rpdc}. 

To achieve nice practical speed-up performance, we make the following assumption:
\setcounter{assumption}{-1}
\begin{assumption}\label{assump:expensive-fi}
The cost of computing $\nabla_i f(\vx)$ is roughly at least $p-1$ times of that of updating $\vx_i$, $\vr$, and $\vlam$ respectively by \eqref{eq:async-update-x}, \eqref{eq:update-r} and \eqref{eq:update-lam} for all $i$, where $p$ is the number of nodes.
 \end{assumption} 
Note that our theoretical analysis does not require this assumption. Roughly speaking, the above assumption means that the worker nodes compute block gradients no faster than the master node can use them. When it holds, the master node can quickly digest the block gradient information fed by all worker nodes. Without this assumption, Algorithm \ref{alg:async-rpdc} may not perform well in terms of parallel efficiency. For example, if $p>2$, and computing $\nabla_i f(\vx)$ takes similar time as updating $\vx_i$, $\vr$ and $\vlam$, then until the $k$-th iteration, there would be roughly $k(p-2)$ partial gradients that have been sent to but not used by the master node. In this case, a lot of computation will be wasted.   

We make a few remarks on Algorithm \ref{alg:async-rpdc} as follows.
\begin{itemize}
\renewcommand\labelitemi{--}
\item \textbf{Special case:} If there is only one node (i.e., $p=1$), the algorithm simply reduces to the non-parallel Algorithm \ref{alg:rpdc}. In this case, Assumption \ref{assump:expensive-fi} trivially holds.
\item \textbf{Iteration number:} Only the master node increases the iteration number $k$, which counts the times $\vlam$ is updated and also the number of used block gradients. The sync-parallel method (e.g., in \cite{GXZ-RPDCU2016}) chooses to update multiple blocks every time, and the computation is distributed over multiple nodes. It generally requires larger weight in the proximal term for convergence. Hence, even if $\vv^k=\nabla_{i_k}f(\vx^k),\,\forall k$, Algorithm \ref{alg:async-rpdc} does not reduce to its sync-parallel counterpart. 
\item \textbf{Delayed information:} Since all worker nodes provide block gradients to the master node, we cannot guarantee every computed block gradient will be immediately used to update $\vx$. Hence, in \eqref{eq:async-update-x}, $\vv^k$ may not equal $\nabla_{i}f(\vx^k)$ but can be a delayed (i.e., outdated) block gradient. The delay is usually in the same order of $p$ and can affect the stepsize, but the affect is negligible as the block number $m$ is greater than the delay in an order (see Theorem \ref{thm:async-cvg-prob}).

Because $\vx$-blocks are computed in the master node, the values of $\vr$ and $\vlam$ used in the update are always up-to-date. One can let worker nodes compute new $\vx_i$'s and then feed them (or also the changes in $\vr$) to the master node. That way, $\vr$ and $\vlam$ will also be outdated when computing $\vx$-blocks. 

\item \textbf{Load balance:} Under Assumption \ref{assump:expensive-fi}, if \eqref{eq:async-update-x} is easy to solve (e.g., $\vP_i=\eta_i \vI$) and all nodes have similar computing power, the master node will have used all received block gradients before a new one comes. We let the master node itself also compute block gradient if there is no new one sent from any worker node. This way, all nodes work continuously without idle wait. Compared to its sync-parallel counterpart that typically suffers serious load imbalance, the async-parallel can achieve better speed-up; see the numerical results in section \ref{sec:svm}. 
\end{itemize}

\begin{algorithm}\caption{Async-parallel randomized primal-dual block update for \eqref{eq:mb-prob}}\label{alg:async-rpdc}
\DontPrintSemicolon
\textbf{Initialization:} choose $\vx^0$ and $\vlam^0=\vzero$; let $\vr^0=\vA\vx^0-\vb$ and $k=0$; set $\beta,\rho$ and $\vP_i$'s.\;
\While{the stopping conditions not satisfied}{
\If{worker node}{
Pick $j$ from $[m]$ uniformly at random.\;
Read $\vx$ from the memory as $\hat{\vx}$.\;
Compute $\nabla_j f(\hat{\vx})$ and send it together with the block $j$ to \emph{master node}
}
\If{master node}{
\If{received one new pair $\big(j, \nabla_{j} f(\hat{\vx})\big)$}{
Let $i_k = j$ and $\vv^k = \nabla_{j} f(\hat{\vx})$}
\Else{
Pick $i_k$ from $[m]$ uniformly at random and let $\vv^k=\nabla_{i_k} f(\vx^k)$
}
For any $i\neq i_k$, keep $\vx_i^{k+1}=\vx^k_i$, and for $i=i_k$, update $\vx_i$ by
\begin{equation}
\vx_i^{k+1}\in\argmin_{\vx_i}\langle \vv^k-\vA_i^\top(\vlam^k-\beta \vr^k), \vx_i\rangle +g_i(\vx_i)+\frac{1}{2}\|\vx_i-\vx_i^k\|_{\vP_i}^2,\label{eq:async-update-x}
\end{equation}Update residual $\vr$ and multipliers $\vlam$ by \eqref{eq:update-r} and \eqref{eq:update-lam}.\;
Let $k\gets k+1$.
}
}
\end{algorithm}

\section{Convergence analysis}\label{sec:analysis}
In this section, we present convergence results of the proposed algorithm. First, we analyze the non-parallel Algorithm \ref{alg:rpdc}. We show that the objective value $F(\vx^k)$ and the residual $\vA\vx^k-\vb$ converge to the optimal value and zero respectively in probability. In addition, we establish a sublinear convergence rate result based on an averaged point. Then, through bounding a cross term involving the delayed block gradient, we establish similar results for the async-parallel Algorithm \ref{alg:async-rpdc}.

Throughout our analysis, we make the following assumptions.
\begin{assumption}[Existence of a solution]\label{assump:sdl-pt}
There exists one pair of primal-dual solution $(\vx^*,\vlam^*)$ such that $\vA\vx^*=\vb$ and $\Phi(\vx,\vx^*,\vlam^*)\ge 0,\,\forall \vx$.
\end{assumption}

\begin{assumption}[Gradient Lipschitz continuity]\label{assump:lip}
There exist constants $L_i$'s and $L_r$ such that for any $\vx$ and $\vy$,
$$\|\nabla_i f(\vx+\vU_i \vy)-\nabla_i f(\vx)\|\le L_i\|\vy_i\|,\,i=1,\ldots,m,$$
and 
$$\|\nabla f(\vx+\vU_i \vy)-\nabla f(\vx)\|\le L_r\|\vy_i\|,\,i=1,\ldots,m.$$
\end{assumption}
Denote $\vL=\diag(L_1,\ldots,L_m)$. Then under the above assumption, it holds that
\begin{equation}\label{eq:lip-prop}
f(\vx+\vU_i \vy)\le f(\vx) + \langle \nabla_i f(\vx), \vy_i\rangle + \frac{L_i}{2}\|\vy_i\|^2,\,\forall i,\,\forall \vx, \vy.
\end{equation}

\subsection{Convergence results of Algorithm \ref{alg:rpdc}}\label{sec:convg-serial}

Although Algorithm \ref{alg:rpdc} is a special case of the method in \cite{GXZ-RPDCU2016}, its convergence analysis is easier and can be made more succinct. In addition, our analysis for Algorithm \ref{alg:async-rpdc} is based on that for Algorithm \ref{alg:rpdc}. Hence, we provide a complete convergence analysis for Algorithm \ref{alg:rpdc}. First, we establish several lemmas, which will be used to show our main convergence results.
\begin{lemma}\label{lem:f-term}
Let $\{\vx^k\}$ be the sequence generated from Algorithm \ref{alg:rpdc}. Then for any $\vx$ independent of $i_k$, it holds that
{\small\begin{align*}
&\EE_{i_k}\left\langle\nabla_{i_k} f(\vx^k),\vx_{i_k}^{k+1}-\vx_{i_k}\right\rangle
\ge  -\big(1-\frac{1}{m}\big)[f(\vx^k)-f(\vx)]+\EE_{i_k}\left[f(\vx^{k+1})-f(\vx)-\frac{1}{2}\|\vx^{k+1}-\vx^k\|_\vL^2\right].
\end{align*}
}
\end{lemma}
\begin{proof}
We write $\langle\nabla_{i_k} f(\vx^k),\vx_{i_k}^{k+1}-\vx_{i_k}\rangle=\langle\nabla_{i_k} f(\vx^k),\vx_{i_k}^{k}-\vx_{i_k}\rangle+\langle\nabla_{i_k} f(\vx^k),\vx_{i_k}^{k+1}-\vx_{i_k}^k\rangle$. For the first term, we use the uniform distribution of $i_k$ on $[m]$ and the convexity of $f$ to have
$$\EE_{i_k}\langle\nabla_{i_k} f(\vx^k),\vx_{i_k}^{k}-\vx_{i_k}\rangle=\frac{1}{m}\langle\nabla f(\vx^k),\vx^{k}-\vx\rangle\ge \frac{1}{m}\big[f(\vx^k)-f(\vx)\big],$$
and for the second term, we use \eqref{eq:lip-prop} to have
\begin{align}\label{eq:crs-t2}
\langle\nabla_{i_k} f(\vx^k),\vx_{i_k}^{k+1}-\vx_{i_k}^k\rangle\ge &f(\vx^{k+1})-f(\vx^k)-\frac{L_{i_k}}{2}\|\vx_{i_k}^{k+1}-\vx_{i_k}^k\|^2\cr
=&f(\vx^{k+1})-f(\vx^k) - \frac{1}{2}\|\vx^{k+1}-\vx^k\|_\vL^2.
\end{align} 
Combining the above two inequalities gives the desired result.
\end{proof}

\begin{lemma}\label{lem:rlam-term}
For any $\vx$ independent of $i_k$ such that $\vA\vx=\vb$, it holds
\begin{align*}
& \EE_{i_k}\langle -\vA_{i_k}^\top(\vlam^k-\beta \vr^k), \vx_{i_k}^{k+1}-\vx_{i_k}\rangle \cr
= & -\big(1-\frac{1}{m}\big)\left(-\langle\vlam^k, \vr^k\rangle + \beta\|\vr^k\|^2\right) -\EE_{i_k}\langle \vlam^{k+1}, \vr^{k+1}\rangle + (\beta-\rho)\EE_{i_k}\|\vr^{k+1}\|^2\cr
& -\frac{\beta}{2} \EE_{i_k}\left[\| \vr^{k+1}\|^2-\| \vr^k\|^2+\|\vx^{k+1}-\vx^k\|_{\vA^\top \vA}^2 \right].
\end{align*}
\end{lemma}
\begin{proof}
Let $\vy^k=-\vA^\top (\vlam^k-\beta \vr^k)$. Then
\begin{align}\label{eq:rlam-term-eq1}
\EE_{i_k}\langle \vy_{i_k}^k, \vx_{i_k}^{k+1}-\vx_{i_k}\rangle =&  \EE_{i_k}\langle \vy_{i_k}^k, \vx_{i_k}^{k}-\vx_{i_k}\rangle+\EE_{i_k}\langle \vy_{i_k}^k, \vx_{i_k}^{k+1}-\vx_{i_k}^k\rangle\cr
=&\frac{1}{m} \langle \vy^k, \vx^{k}-\vx\rangle+ \EE_{i_k}\langle \vy^k, \vx^{k+1}-\vx^k\rangle\cr
=&-\big(1-\frac{1}{m}\big) \langle \vy^k, \vx^{k}-\vx\rangle+ \EE_{i_k}\langle \vy^k, \vx^{k+1}-\vx\rangle.
\end{align}
Note $\vy^k=-\vA^\top \vlam^{k+1}+(\beta-\rho) \vA^\top \vr^{k+1}-\beta \vA^\top(\vr^{k+1}-\vr^k)$ and $\vr^{k+1}-\vr^k=\vA(\vx^{k+1}-\vx^k)$. In addition, from $\vA\vx=\vb$, we have $\vA(\vx^{k+1}-\vx)=\vr^{k+1}$. Hence,
\begin{align}\label{eq:rlam-term-eq2}
\langle \vy^k, \vx^{k+1}-\vx\rangle = \langle -\vA^\top \vlam^{k+1}, \vx^{k+1}-\vx\rangle + (\beta-\rho)\|\vr^{k+1}\|^2-\beta\big\langle \vA(\vx^{k+1}-\vx^k), \vA(\vx^{k+1}-\vx)\big\rangle.
\end{align}
Noting
$$\big\langle \vA(\vx^{k+1}-\vx^k), \vA(\vx^{k+1}-\vx)\big\rangle=\frac{1}{2}\big[\| \vr^{k+1}\|^2-\| \vr^k\|^2+\|\vx^{k+1}-\vx^k\|_{\vA^\top \vA}^2 \big],$$
we complete the proof by plugging \eqref{eq:rlam-term-eq2} into \eqref{eq:rlam-term-eq1}.
\end{proof}

\begin{lemma}\label{lem:g-term}
For any $\vx$ independent of $i_k$, it holds
\begin{align*}
\EE_{i_k}\left\langle\tilde{\nabla}g_{i_k}(\vx_{i_k}^{k+1}), \vx_{i_k}^{k+1}-\vx_{i_k}\right\rangle\ge \EE_{i_k}[g(\vx^{k+1})-g(\vx)] - \big(1-\frac{1}{m}\big)[g(\vx^k)-g(\vx)],
\end{align*}
where $\tilde{\nabla}g_{i_k}(\vx_{i_k}^{k+1})$ denotes a subgradient of $g_{i_k}$ at $\vx_{i_k}^{k+1}$.
\end{lemma}
\begin{proof}
From the convexity of $g_{i_k}$ and definition of subgradient, it follows that
\begin{equation}\label{eq:gik-ineq}
\EE_{i_k}\left\langle\tilde{\nabla}g_{i_k}(\vx_{i_k}^{k+1}), \vx_{i_k}^{k+1}-\vx_{i_k}\right\rangle\ge \EE_{i_k}\big[g_{i_k}(\vx_{i_k}^{k+1})- g_{i_k}(\vx_{i_k})\big].
\end{equation}
Writing $g_{i_k}(\vx_{i_k}^{k+1})- g_{i_k}(\vx_{i_k})=g_{i_k}(\vx_{i_k}^{k})- g_{i_k}(\vx_{i_k})+g_{i_k}(\vx_{i_k}^{k+1})- g_{i_k}(\vx_{i_k}^k)$ and taking the conditional expectation give
$$\EE_{i_k}\big[g_{i_k}(\vx_{i_k}^{k+1})- g_{i_k}(\vx_{i_k})\big]=\frac{1}{m}\big[g(\vx^k)-g(\vx)\big]+\EE_{i_k}\big[g(\vx^{k+1})-g(\vx^k)\big].$$
We obtain the desired result by plugging the above equation into \eqref{eq:gik-ineq}.
\end{proof}
Using the above three lemmas, we show an inequality after each iteration of the algorithm. 
\begin{theorem}[Fundamental result]
Let $\{(\vx^k,\vr^k,\vlam^k)\}$ be the sequence generated from Algorithm \ref{alg:rpdc}. Then for any $\vx$ such that $\vA\vx=\vb$, it holds
\begin{align}\label{eq:1iter-result}
&\EE_{i_k}\left[F(\vx^{k+1})-F(\vx) - \langle \vlam^{k+1}, \vr^{k+1}\rangle+(\beta-\rho)\|\vr^{k+1}\|^2-\frac{\beta}{2}\| \vr^{k+1}\|^2\right]\cr
&+\frac{1}{2}\EE_{i_k}\left[\|\vx^{k+1}-\vx\|_\vP^2+\|\vx^{k+1}-\vx^k\|_{\vP-\vL-\beta\vA^\top\vA}^2\right]\cr
\le& \big(1-\frac{1}{m}\big)\left[F(\vx^k)-F(\vx)-\langle \vlam^k, \vr^k\rangle+\beta\|\vr^k\|^2\right] -\frac{\beta}{2}\|\vr^k\|^2+\frac{1}{2}\|\vx^k-\vx\|_\vP^2,
\end{align}
where $\vP=\mathrm{blkdiag}(\vP_1,\ldots,\vP_m)$.
\end{theorem}

\begin{proof}
Since $\vx_{i_k}^{k+1}$ is one solution to \eqref{eq:update-x}, there is a subgradient $\tilde{\nabla}g_{i_k}(\vx_{i_k}^{k+1})$ of $g_{i_k}$ at $\vx_{i_k}^{k+1}$ such that
\begin{align*}
&\nabla_{i_k} f(\vx^k)-\vA_{i_k}^\top(\vlam^k-\beta \vr^k)+ \tilde{\nabla}g_{i_k}(\vx_{i_k}^{k+1})+\vP_{i_k}(\vx_{i_k}^{k+1}-\vx_{i_k}^k)=0,
\end{align*}
Hence,
\begin{align}\label{eq:sub-opt}
\EE_{i_k}\Big\langle\nabla_{i_k} f(\vx^k)-\vA_{i_k}^\top(\vlam^k-\beta \vr^k)+ \tilde{\nabla}g_{i_k}(\vx_{i_k}^{k+1})+\vP_{i_k}(\vx_{i_k}^{k+1}-\vx_{i_k}^k),\vx_{i_k}^{k+1}-\vx_{i_k}\Big\rangle=0.
\end{align}
In the above equation, using Lemmas \ref{lem:f-term} through \ref{lem:g-term} and noting 
\begin{equation}\label{eq:term-p}\Big\langle \vP_{i_k}(\vx_{i_k}^{k+1}-\vx_{i_k}^k),\vx_{i_k}^{k+1}-\vx_{i_k}\Big\rangle=\frac{1}{2}\left[\|\vx^{k+1}-\vx\|_\vP^2-\|\vx^k-\vx\|_\vP^2+\|\vx^{k+1}-\vx^k\|_\vP^2\right],
\end{equation}
we have the desired result.
\end{proof}

Now we are ready to show the convergence results of Algorithm \ref{alg:rpdc}.
\begin{theorem}[Global convergence in probability]\label{thm:cvg-prob}
Let $\{(\vx^k,\vr^k,\vlam^k)\}$ be the sequence generated from Algorithm \ref{alg:rpdc}. If $0<\rho\le\frac{\beta}{m}$ and $\vP_i\succeq L_i\vI+\beta \vA_i^\top \vA_i,\,\forall i$, then
$$F(\vx^k)\overset{p}\to F(\vx^*),\text{ and } \|\vr^k\|\overset{p}\to 0.$$
\end{theorem}
Before proving the theorem, we make a remark here.
The dual stepsize $\rho$ can be up to $\frac{\beta}{m}$, so it could be much smaller than $\beta$ as $m$ is big. However, note that $\vlam$ is renewed more frequently than $\vx$. It is updated once immediately after one change to $\vx$. Hence, if $\rho=\frac{\beta}{m}$, after one epoch of $\vx$-update, the dual variable $\vlam$ has been updated $m$ times and moved a step of size $\beta$. That is why we can still observe fast convergence of the algorithm to the optimal solution even though a small $\rho$ is used; see the numerical results in section \ref{sec:numerical}. 
\begin{proof}
 Note that
$$F(\vx^k)-F(\vx)-\langle \vlam^k, \vr^k\rangle=\Phi(\vx^k,\vx,\vlam)+\langle\vlam-\vlam^k,\vr^k\rangle.$$
Hence, taking expectation over both sides of \eqref{eq:1iter-result} and summing up from $k=0$ through $K$ yield
\begin{align}\label{eq:1iter-result-exp}
&\EE\left[\Phi(\vx^{K+1},\vx,\vlam)+\langle\vlam-\vlam^{K+1},\vr^{K+1}\rangle\right]+\frac{1}{m}\sum_{k=1}^K\EE\left[\Phi(\vx^{k},\vx,\vlam)+\langle\vlam-\vlam^{k},\vr^{k}\rangle\right]+(\beta-\rho)\EE\|\vr^{K+1}\|^2\nonumber\\
&+\big(\frac{\beta}{m}-\rho\big)\sum_{k=1}^K\|\vr^k\|^2-\frac{\beta}{2}\EE\|\vr^{K+1}\|^2+\frac{1}{2}\EE\|\vx^{K+1}-\vx\|_\vP^2+\frac{1}{2}\sum_{k=0}^K\EE\|\vx^{k+1}-\vx^k\|_{\vP-\vL-\beta \vA^\top \vA}^2\nonumber\\
\le& \big(1-\frac{1}{m}\big)\left[F(\vx^0)-F(\vx)-\langle \vlam^0, \vr^0\rangle+\beta\|\vr^0\|^2\right]+\frac{1}{2}\|\vx^0-\vx\|_\vP^2-\frac{\beta}{2}\|\vr^0\|^2.
\end{align}
Since $\vlam^{K+1}=\vlam^K-\rho \vr^{K+1}$, it follows from Young's inequality that
\begin{equation}\label{eq:1iter-result-exp-lam-K+1}
\langle\vlam-\vlam^{K+1},\vr^{K+1}\rangle+(\beta-\rho)\|\vr^{K+1}\|^2-\frac{\beta}{2}\|\vr^{K+1}\|^2\ge - \frac{1}{2\beta}\|\vlam-\vlam^K\|^2.
\end{equation}
In addition,
\begin{align}\label{eq:1iter-result-exp-lam-k}
\sum_{k=1}^K\langle\vlam-\vlam^{k},\vr^{k}\rangle=&\frac{1}{2\rho}\sum_{k=1}^K\big[\|\vlam-\vlam^{k}\|^2-\|\vlam-\vlam^{k-1}\|^2+\|\vlam^{k-1}-\vlam^{k}\|^2\big]\cr
=&\frac{1}{2\rho}\left[\|\vlam-\vlam^{K}\|^2-\|\vlam-\vlam^{0}\|^2\right]+\frac{\rho}{2}\sum_{k=1}^K\|\vr^{k}\|^2.
\end{align}
Plugging \eqref{eq:1iter-result-exp-lam-K+1} and \eqref{eq:1iter-result-exp-lam-k} into \eqref{eq:1iter-result-exp} and using $\vlam^0=0$, we have
\begin{align}\label{eq:1iter-result-exp2}
&\EE\Phi(\vx^{K+1},\vx,\vlam)+\frac{1}{m}\sum_{k=1}^K\EE\Phi(\vx^{k},x,\vlam)+\big(\frac{\beta}{m}+\frac{\rho}{2m}-\rho\big)\sum_{k=1}^K\EE\|\vr^k\|^2+\big(\frac{1}{2m\rho}-\frac{1}{2\beta}\big)\EE\|\vlam-\vlam^K\|^2\nonumber\\
&+\frac{1}{2}\EE\|\vx^{K+1}-\vx\|_\vP^2+\frac{1}{2}\sum_{k=0}^K\EE\|\vx^{k+1}-\vx^k\|_{\vP-\vL-\beta \vA^\top \vA}^2\nonumber\\
\le& \big(1-\frac{1}{m}\big)\left[F(\vx^0)-F(\vx)+\beta\|\vr^0\|^2\right]+\frac{1}{2}\|\vx^0-x\|_\vP^2-\frac{\beta}{2}\|\vr^0\|^2+\frac{1}{2m\rho}\EE\|\vlam\|^2.
\end{align}

Letting $(\vx,\vlam)=(\vx^*,\vlam^*)$ in the above equality, we have from $\vP_i\succeq L_i\vI + \beta \vA_i^\top \vA_i$ and $\beta\ge m\rho$ that
$$\frac{1}{m}\sum_{k=1}^K\EE\Phi(\vx^{k},\vx^*,\vlam^*)+\big(\frac{\beta}{m}+\frac{\rho}{2m}-\rho\big)\sum_{k=1}^K\EE\|\vr^k\|^2 < \infty,\,\forall K,$$
which together with $|\EE \xi|^2 \le \EE\xi^2$ implies that
\begin{subequations}\label{eq:convg-expect}
\begin{align}
&\lim_{k\to\infty} \EE\Phi(\vx^{k},\vx^*,\vlam^*) = 0,\\
&\lim_{k\to\infty} \EE\|\vr^k\|=0.
\end{align}
\end{subequations}

For any $\epsilon>0$, it follows from the Markov's inequality that
$$\Prob(\|\vr^k\|>\epsilon)\le \frac{\EE\|\vr^k\|}{\epsilon}\to 0, \text{ as }k\to\infty,$$
and
\begin{align}\label{eq:prob-F-eps}
&\Prob(|F(\vx^k)-F(\vx^*)|\ge \epsilon)\cr
=&\Prob(F(\vx^k)-F(\vx^*)\ge \epsilon)+\Prob(F(\vx^k)-F(\vx^*)\le -\epsilon)\cr
\le & \Prob(F(\vx^k)-F(\vx^*)-\langle\vlam^*, \vr^k\rangle\ge\frac{\epsilon}{2}) + \Prob(\langle\vlam^*, \vr^k\rangle\ge\frac{\epsilon}{2})+\Prob(-\langle\vlam^*, \vr^k\rangle\ge\epsilon)\cr
\le & \Prob(F(\vx^k)-F(\vx^*)-\langle\vlam^*, \vr^k\rangle\ge\frac{\epsilon}{2}) + \Prob(\|\vlam^*\|\cdot \|\vr^k\|\ge\frac{\epsilon}{2})+\Prob(\|\vlam^*\|\cdot \|\vr^k\|\ge\epsilon)\cr
\to& 0,\text{ as }k\to\infty,
\end{align}
where in the first inequality, we have used the fact $F(\vx)-F(\vx^*)-\langle\vlam^*, \vA\vx-\vb\rangle \ge 0,\,\forall \vx$, and the last equation follows from \eqref{eq:convg-expect} and the Markov's inequality. This completes the proof.
\end{proof}

Given any $\epsilon>0$ and $\sigma\in(0,1)$, we can also estimate the number of iterations for the algorithm to produce a solution satisfying an error bound $\epsilon$ with probability no less than $1-\sigma$.
\begin{definition}[$(\epsilon,\sigma)$-solution]
Given $\epsilon>0$ and $0<\sigma<1$, a random vector $\vx$ is called an $(\epsilon,\sigma)$-solution to \eqref{eq:mb-prob} if
$\Prob(|F(\vx)-F(\vx^*)|\ge\epsilon) \le \sigma$ and $\Prob(\|\vA\vx-\vb\|\ge\epsilon)\le \sigma.$
\end{definition}

\begin{theorem}[Ergodic convergence rate]\label{thm:rate}
Let $\{(\vx^k,\vr^k,\vlam^k)\}$ be the sequence generated from Algorithm \ref{alg:rpdc}. Assume $0<\rho\le\frac{\beta}{m}$ and $\vP_i\succeq L_i\vI+\beta \vA_i^\top \vA_i,\,\forall i$. Let $\bar{\vx}^{K+1}=\frac{\vx^{K+1}+\sum_{k=1}^K \vx^{k+1}/m}{1+K/m}$ and 
$$C_0=\big(1-\frac{1}{m}\big)\left[F(\vx^0)-F(\vx^*)\right]+\frac{1}{2}\|\vx^0-\vx^*\|_\vP^2+\big(\frac{\beta}{2}-\frac{\beta}{m}\big)\|\vr^0\|^2.$$
Then
\begin{align}
&-\frac{1}{1+K/m}\left(C_0+\frac{2}{m\rho}\|\vlam^*\|^2\right)\le\EE[F(\bar{\vx}^{K+1})-F(\vx^*)] \le  \frac{C_0}{1+K/m},\label{eq:rate-obj}\\
&\EE \|\vA\bar{\vx}^{K+1}-\vb\|\le \frac{1}{1+K/m}\left(C_0+\frac{1}{2m\rho}(1+\|\vlam^*\|)^2\right).\label{eq:rate-feas}
\end{align}
In addition, given any $\epsilon>0$ and $0<\sigma<1$, if 
\begin{equation}\label{eq:def-K}
K\ge m\cdot\max\left(\frac{C_0+\frac{1}{2m\rho}(1+\|\vlam^*\|)^2}{\epsilon\sigma}-1,\, \frac{5C_0 + \frac{13}{2m\rho}\|\vlam^*\|^2}{\epsilon\sigma}-1\right),
\end{equation}
then $\bar{\vx}^{K+1}$ is an $(\epsilon, \sigma)$-solution to \eqref{eq:mb-prob}. 
\end{theorem}
\begin{proof}
Since $F$ is convex, it follows from \eqref{eq:1iter-result-exp2} that
\begin{equation}\label{eq:rate-any-lam}
\EE\Phi(\bar{\vx}^{K+1}, \vx, \vlam) \le \frac{1}{1+K/m}\left(C_0+\frac{1}{2m\rho}\EE\|\vlam\|^2\right),
\end{equation}
which with $\vx=\vx^*$ and $\vlam=0$ implies the second inequality in \eqref{eq:rate-obj}.
From $\Phi(\vx,\vx^*,\vlam^*)\ge 0,\,\forall \vx$ and Cauchy-Schwartz inequality, we have that
\begin{equation}\label{eq:opt-ineq}
F(\vx)-F(\vx^*)\ge -\|\vlam^*\|\cdot\|\vA\vx-\vb\|,\,\forall \vx.
\end{equation}
Letting $\vx=\vx^*$ and $\vlam=-\frac{1+\|\vlam^*\|}{\|\vA\bar{\vx}^{K+1}-\vb\|}(\vA\bar{\vx}^{K+1}-\vb)$ in \eqref{eq:rate-any-lam}  and  using \eqref{eq:opt-ineq} give \eqref{eq:rate-feas}, where we have used the convention $\frac{0}{0}=0$.
By Markov's inequality, 
$$\Prob(\|\vA\bar{\vx}^{K+1}-\vb\|\ge \epsilon)\le \frac{\EE\|\vA\bar{\vx}^{K+1}-\vb\|}{\epsilon},$$
and thus to have $\Prob(\|\vA\bar{\vx}^{K+1}-\vb\|\ge \epsilon)\le \sigma$, it suffices to let 
\begin{equation}\label{eq:K-bd1}
K\ge \frac{C_0+\frac{1}{2m\rho}(1+\|\vlam^*\|)^2}{\epsilon\sigma}m-m.
\end{equation}

Similarly, letting $\vx=\vx^*$ and $\vlam=-\frac{2\|\vlam^*\|}{\|\vA\bar{\vx}^{K+1}-\vb\|}(\vA\bar{\vx}^{K+1}-\vb)$ in \eqref{eq:rate-any-lam}  and  using \eqref{eq:opt-ineq} give
$$\|\vlam^*\|\cdot\EE\|\vA\bar{\vx}^{K+1}-\vb\|\le \frac{1}{1+K/m}\left(C_0+\frac{2}{m\rho}\|\vlam^*\|^2\right),$$
which together with \eqref{eq:opt-ineq} implies the first inequality in \eqref{eq:rate-obj}.
Through the same arguments that show \eqref{eq:prob-F-eps}, we have
\begin{align}\label{eq:prob-F-eps-rate}
&\Prob(|F(\bar{\vx}^{K+1})-F(\vx^*)|\ge \epsilon)\\
\le & \Prob(\Phi(\bar{\vx}^{K+1},\vx^*,\vlam^*)\ge\frac{\epsilon}{2}) + \Prob(\|\vlam^*\|\cdot \|\vA\bar{\vx}^{K+1}-\vb\|\ge\frac{\epsilon}{2})+\Prob(\|\vlam^*\|\cdot \|\vA\bar{\vx}^{K+1}-\vb\|\ge\epsilon)\nonumber\\
\le & \frac{\EE\Phi(\bar{\vx}^{K+1},\vx^*,\vlam^*)}{\epsilon/2}+\frac{\|\vlam^*\|\cdot \EE\|\vA\bar{\vx}^{K+1}-\vb\|}{\epsilon/2}+\frac{\|\vlam^*\|\cdot \EE\|\vA\bar{\vx}^{K+1}-\vb\|}{\epsilon}.\nonumber
\end{align}
Hence, to have $\Prob(|F(\bar{\vx}^{K+1})-F(\vx^*)|\ge \epsilon)\le \sigma$, it suffices to let
$$K\ge \frac{5C_0 + \frac{13}{2m\rho}\|\vlam^*\|^2}{\epsilon\sigma}m-m,$$
which together with \eqref{eq:K-bd1} gives the desired result and thus completes the proof.
\end{proof}

\subsection{Convergence results of Algorithm \ref{alg:async-rpdc}}
The key difference between Algorithms \ref{alg:rpdc} and \ref{alg:async-rpdc} is that $\vv^k$ used in \eqref{eq:async-update-x} may not equal the block gradient of $f$ at $\vx^k$ but another outdated vector, which we denote as $\hat{\vx}^k$. This delayed vector may not be any iterate that ever exists in the memory, i.e., inconsistent read can happen \cite{liu2015async-scd}. Besides Assumptions \ref{assump:sdl-pt} and \ref{assump:lip}, we make an additional assumption on the delayed vector.
\begin{assumption}[Bounded delay]\label{assump:delay}
The delay is uniformly bounded by an integer $\tau$, and $\hat{\vx}^k$ can be related to $\vx^k$ by the equation
\begin{equation}\label{eq:relation-xk-hat}
\hat{\vx}^k=\vx^k+\sum_{d\in J(k)}(\vx^d-\vx^{d+1}),
\end{equation}
where $J(k)$ is a subset of $\{k-\tau, k-\tau+1,\ldots, k-1\}$.
\end{assumption}
 The boundedness of the delay holds if there is no ``dead'' node. The relation between $\vx^k$ and $\hat{\vx}^k$ in \eqref{eq:relation-xk-hat} is satisfied if the read of each block variable is consistent, which can be guaranteed by a \emph{dual memory approach}; see \cite{Peng_2016_AROCK}. 

Similar to \eqref{eq:sub-opt}, we have from the optimality condition of \eqref{eq:async-update-x} that
\begin{align}\label{eq:async-sub-opt}
\EE_{i_k}\Big\langle\nabla_{i_k} f(\hat{\vx}^k)-\vA_{i_k}^\top(\vlam^k-\beta \vr^k)+ \tilde{\nabla}g_{i_k}(\vx_{i_k}^{k+1})+\vP_{i_k}(\vx_{i_k}^{k+1}-\vx_{i_k}^k),\vx_{i_k}^{k+1}-\vx_{i_k}\Big\rangle=0,
\end{align}
where we have used $\vv^k=\nabla_{i_k}f(\hat{\vx}^k)$. Except $\EE_{i_k}\langle \nabla_{i_k}f(\hat{\vx}^k), \vx_{i_k}^{k+1}-\vx_{i_k}\rangle$, all the other terms in \eqref{eq:async-sub-opt} can be bounded in the same ways as those in section \ref{sec:convg-serial}. We first show how to bound this term and then present the convergence results of Algorithm \ref{alg:async-rpdc}.

\begin{lemma}\label{lem:bound-delay-term}
Under Assumptions \ref{assump:lip} and \ref{assump:delay}, we have for any $\alpha>0$ that
\begin{align}
&\EE_{i_k}\langle \nabla_{i_k}f(\hat{\vx}^k), \vx_{i_k}^{k+1}-\vx_{i_k}\rangle\cr
\ge & \EE_{i_k}[f(\vx^{k+1})-f(\vx)] - \big(1-\frac{1}{m}\big)[f(\vx^k)-f(\vx)]-\frac{1}{2}\EE_{i_k}\|\vx^{k+1}-\vx^k\|_{\vL+\alpha L_c\vI}^2\cr
&-\frac{\kappa L_r\tau/\alpha+2 L_r\tau}{2m}\sum_{d=k-\tau}^{k-1}\|\vx^{d+1}-\vx^d\|^2-\frac{1}{2m}\sum_{d=k-\tau}^{k-1}\|\vx^{d+1}-\vx^d\|^2_\vL,
\end{align}
where $L_c=\max_i L_i>0$, and $\kappa = \frac{L_r}{L_c}$ denotes the condition number.
\end{lemma} 
\begin{proof} 
We split $\EE_{i_k}\langle \nabla_{i_k}f(\hat{\vx}^k), \vx_{i_k}^{k+1}-\vx_{i_k}\rangle$ into four terms:
\begin{align}\label{eq:crs-term}
&\EE_{i_k}\langle \nabla_{i_k}f(\hat{\vx}^k), \vx_{i_k}^{k+1}-\vx_{i_k}\rangle\cr
=&\EE_{i_k}\langle \nabla_{i_k}f({\vx}^k), \vx_{i_k}^{k+1}-\vx_{i_k}^k\rangle-\EE_{i_k}\langle \nabla_{i_k}f({\vx}^k)-\nabla_{i_k}f(\hat{\vx}^k), \vx_{i_k}^{k+1}-\vx_{i_k}^k\rangle\cr
&+\EE_{i_k}\langle \nabla_{i_k}f(\hat{\vx}^k), \vx_{i_k}^{k}-\hat{\vx}_{i_k}^k\rangle+\EE_{i_k}\langle \nabla_{i_k}f(\hat{\vx}^k), \hat{\vx}_{i_k}^k-\vx_{i_k}\rangle,
\end{align}
and we bound each of the four cross terms in \eqref{eq:crs-term}. The first is bounded in \eqref{eq:crs-t2}. 
Secondly, from the convexity of $f$, we have
\begin{equation}\label{eq:crs-t1}
\EE_{i_k}\langle \nabla_{i_k}f(\hat{\vx}^k), \hat{\vx}_{i_k}^k-\vx_{i_k}\rangle=\frac{1}{m}\langle \nabla f(\hat{\vx}^k), \hat{\vx}^k-\vx\rangle\ge\frac{1}{m}[f(\hat{\vx}^k)-f(\vx)].
\end{equation}

For the other two terms, we use the relation between $\hat{\vx}^k$ and $\vx^k$ in \eqref{eq:relation-xk-hat}. 
From the result in \cite[pp.306]{liu2015asynchronous-JMLR}, it holds that
\begin{equation}\label{eq:grad-delay}
\|\nabla f({\vx}^k)-\nabla f(\hat{\vx}^k)\|\le L_r\sum_{d\in J(k)}\|\vx^{d+1}-\vx^d\|.
\end{equation}
Hence by Young's inequality, we have for any $\alpha>0$ that
\begin{align}\label{eq:crs-t3}
&-\EE_{i_k}\langle \nabla_{i_k}f({\vx}^k)-\nabla_{i_k}f(\hat{\vx}^k), \vx_{i_k}^{k+1}-\vx_{i_k}^k\rangle\cr
\ge&-\frac{1}{2 \alpha L_c}\EE_{i_k}\| \nabla_{i_k} f({\vx}^k)-\nabla_{i_k} f(\hat{\vx}^k)\|^2 -\frac{\alpha L_c}{2}\EE_{i_k} \|\vx_{i_k}^{k+1}-\vx_{i_k}^k\|^2\cr
=&-\frac{1}{2m \alpha L_c}\| \nabla f({\vx}^k)-\nabla f(\hat{\vx}^k)\|^2- \frac{\alpha L_c}{2}\EE_{i_k}\|\vx^{k+1}-\vx^k\|^2\cr
\overset{\eqref{eq:grad-delay}}\ge& -\frac{ L_r^2|J(k)|}{2m \alpha L_c}\sum_{d\in J(k)}\|\vx^{d+1}-\vx^d\|^2- \frac{\alpha L_c}{2}\EE_{i_k}\|\vx^{k+1}-\vx^k\|^2.
\end{align}

Let $\tau_k=|J(k)|$ and order the elements in $J(k)$ as $d_1<d_2<\ldots<d_{\tau_k}$. Define $\hat{\vx}^{k,0}=\hat{\vx}^k$ and
$\hat{\vx}^{k,j}=\hat{\vx}^k+\sum_{t=1}^j(\vx^{d_t+1}-\vx^{d_t}),\, j=1,\ldots,\tau_k.$
Then we have
\begin{align}\label{eq:crs-t4-1}
&\EE_{i_k}\langle \nabla_{i_k}f(\hat{\vx}^k), \vx_{i_k}^{k}-\hat{\vx}_{i_k}^k\rangle\cr
=&\frac{1}{m} \langle \nabla f(\hat{\vx}^k), \vx^{k}-\hat{\vx}^k\rangle\cr
=&\frac{1}{m} \sum_{j=0}^{\tau_k-1}\langle \nabla f(\hat{\vx}^k), \hat{\vx}^{k,j+1}-\hat{\vx}^{k,j}\rangle\cr
=&\frac{1}{m} \sum_{j=0}^{\tau_k-1}\left[\langle \nabla f(\hat{\vx}^{k,j}), \hat{\vx}^{k,j+1}-\hat{\vx}^{k,j}\rangle-\langle \nabla f(\hat{\vx}^{k,j})-\nabla f(\hat{\vx}^k), \hat{\vx}^{k,j+1}-\hat{\vx}^{k,j}\rangle\right].
\end{align}
Since $\hat{\vx}^{k,j+1}-\hat{\vx}^{k,j}=\vx^{d_{j+1}+1}-\vx^{d_{j+1}}$, it follows from \eqref{eq:lip-prop} that
\begin{equation}\label{eq:crs-t4-2}
\langle \nabla f(\hat{\vx}^{k,j}), \hat{\vx}^{k,j+1}-\hat{\vx}^{k,j}\rangle\ge f(\hat{\vx}^{k,j+1})-f(\hat{\vx}^{k,j})-\frac{1}{2}\|\vx^{d_{j+1}+1}-\vx^{d_{j+1}}\|_\vL^2.
\end{equation}
Note $\nabla f(\hat{\vx}^{k,j})-\nabla f(\hat{\vx}^k)=\sum_{t=0}^{j-1}(\nabla f(\hat{\vx}^{k,t+1})-\nabla f(\hat{\vx}^{k,t}))$. Thus, by the Cauchy-Schwarz inequality and the Young's inequality, we have
\begin{align}\label{eq:crs-t4-3}
&\langle \nabla f(\hat{\vx}^{k,j})-\nabla f(\hat{\vx}^k), \hat{\vx}^{k,j+1}-\hat{\vx}^{k,j}\rangle\cr
\le & \sum_{t=0}^{j-1}\|\nabla f(\hat{\vx}^{k,t+1})-\nabla f(\hat{\vx}^{k,t})\|\cdot\|\hat{\vx}^{k,j+1}-\hat{\vx}^{k,j}\|\cr
\le & L_r\sum_{t=0}^{j-1}\|\hat{\vx}^{k,t+1}-\hat{\vx}^{k,t}\|\cdot\|\hat{\vx}^{k,j+1}-\hat{\vx}^{k,j}\|\cr
\le & \frac{L_r}{2}\sum_{t=0}^{j-1}\left(\|\hat{\vx}^{k,t+1}-\hat{\vx}^{k,t}\|^2+\|\hat{\vx}^{k,j+1}-\hat{\vx}^{k,j}\|^2\right).
\end{align}
Plugging \eqref{eq:crs-t4-2} and \eqref{eq:crs-t4-3} into \eqref{eq:crs-t4-1} gives
\begin{align}\label{eq:crs-t4}
&\EE_{i_k}\langle \nabla_{i_k}f(\hat{\vx}^k), \vx_{i_k}^{k}-\hat{\vx}_{i_k}^k\rangle\\
\ge&\frac{1}{m}\left[f(\vx^k)-f(\hat{\vx}^k)-\frac{1}{2}\sum_{d\in J(k)}\|\vx^{d+1}-\vx^d\|_\vL^2\right]-\frac{L_r}{2m}\sum_{j=0}^{\tau_k-1}\left(\sum_{t=0}^{j-1}\|\hat{\vx}^{k,t+1}-\hat{\vx}^{k,t}\|^2+j\|\hat{\vx}^{k,j+1}-\hat{\vx}^{k,j}\|^2\right)\nonumber
\end{align}
Noting $\tau_k\le \tau$, we have the desired result by 
plugging \eqref{eq:crs-t2}, \eqref{eq:crs-t1}, \eqref{eq:crs-t3}, and \eqref{eq:crs-t4} into \eqref{eq:crs-term}.
\end{proof}

From Lemmas \ref{lem:rlam-term}, \ref{lem:g-term}, and \ref{lem:bound-delay-term}, and also the equation \eqref{eq:term-p}, we can easily have the following result. 
\begin{align}\label{eq:one-iter-async}
&\EE_{i_k}\left[F(\vx^{k+1})-F(\vx)-\langle\vlam^{k+1}, \vr^{k+1}\rangle + (\beta-\rho)\|\vr^{k+1}\|^2-\frac{\beta}{2}\| \vr^{k+1}\|^2+\frac{1}{2}\|\vx^{k+1}-\vx\|_\vP^2\right]\cr
&+\frac{1}{2}\EE_{i_k}\|\vx^{k+1}-\vx^k\|_{\vP-\vL-\alpha L_c\vI-\beta \vA^\top \vA}^2-\frac{\kappa L_r\tau/\alpha+2 L_r\tau}{2m}\sum_{d=k-\tau}^{k-1}\|\vx^{d+1}-\vx^d\|^2-\frac{1}{2m}\sum_{d=k-\tau}^{k-1}\|\vx^{d+1}-\vx^d\|^2_\vL\cr
\le & \big(1-\frac{1}{m}\big)\left[F(\vx^k)-F(\vx)-\langle\vlam^k, \vr^k\rangle + \beta\|\vr^k\|^2\right] -\frac{\beta}{2}\|\vr^k\|^2+\frac{1}{2}\|\vx^k-\vx\|_\vP^2.
\end{align}
Regard $\vx^k=\vx^0,\,\forall k<0$. Hence,
$$\sum_{k=0}^K\sum_{d=k-\tau}^{k-1}\|\vx^{d+1}-\vx^d\|^2\le \tau\sum_{k=0}^{K-1}\|\vx^{k+1}-\vx^k\|^2.$$
Using \eqref{eq:one-iter-async} and following the same arguments in the proofs of Theorems \ref{thm:cvg-prob} and \ref{thm:rate}, we obtain the two theorems below.

\begin{theorem}[Global convergence in probability]\label{thm:async-cvg-prob}
Let $\{(\vx^k,\vr^k,\vlam^k)\}$ be the sequence generated from Algorithm \ref{alg:async-rpdc} with $0<\rho\le\frac{\beta}{m}$ and $\vP_i$'s satisfying
 \begin{equation}\label{eq:async-P}\vP_i\succeq \left(L_i+\alpha L_c+\frac{\tau L_i}{m}+\frac{(\kappa/\alpha+2) L_r\tau^2}{m}\right)\vI+\beta \vA_i^\top\vA_i,\,i=1,\ldots,m,
\end{equation}
for $\alpha>0$, then
$$F(\vx^k)\overset{p}\to F(\vx^*),\quad \|\vr^k\|\overset{p}\to 0.$$
\end{theorem}

\begin{theorem}[Ergodic convergence rate]\label{thm:async-rate}
Under the assumptions of Theorem \ref{thm:async-cvg-prob}, let $\bar{\vx}^{K+1}=\frac{\vx^{K+1}+\sum_{k=1}^K \vx^{k+1}/m}{1+K/m}$ and 
$$C_0=\big(1-\frac{1}{m}\big)\left[F(\vx^0)-F(\vx^*)\right]+\frac{1}{2}\|\vx^0-\vx^*\|_\vP^2+\big(\frac{\beta}{2}-\frac{\beta}{m}\big)\|\vr^0\|^2.$$
Then we have the same results as those in \eqref{eq:rate-obj} and \eqref{eq:rate-feas}. 
In addition, given any $\epsilon>0$ and $0<\sigma<1$, if $K$ satisfies \eqref{eq:def-K},  
then $\bar{\vx}^{K+1}$ is an $(\epsilon, \sigma)$-solution to \eqref{eq:mb-prob}. 
\end{theorem}

\begin{remark}
Comparing the settings of $\vP_i$'s in Theorems \ref{thm:cvg-prob} and \ref{thm:async-cvg-prob}, we see that they are only weakly affected by the delay if $\tau=o(\sqrt{m})$, which holds for problems involving extremely many variables. If all $p$ nodes compute at the same rate, $\tau$ is in the same order of $p$ \cite{peng2016-total-async}, and thus Theorem \ref{thm:async-cvg-prob} indicates that nearly linear speed-up can be achieved on $O(\sqrt{m})$ nodes. Even without the nonseparable affine constraint, this quantity is better than that required in \cite{liu2015async-scd}. In addition, as $\tau=0$, Algorithm \ref{alg:async-rpdc} reduces to Algorithm \ref{alg:rpdc}, and their convergence results coincide. 
\end{remark}

\section{Numerical experiments}\label{sec:numerical}
In this section, we test the proposed methods on the basis pursuit problem \eqref{eq:bp}, the nonnegativity constrained quadratic programming, and also the dual SVM \eqref{eq:dsvm}. We demonstrate their efficacy by comparing to several other existing algorithms. 

\subsection{Basis pursuit}\label{sec:bp}
The tests in this subsection compare Algorithm \ref{alg:rpdc} to the linearized augmented Lagrangian method (LALM) and the open-source solver YALL1 \cite{YALL12011} on the basis pursuit problem \eqref{eq:bp}. Putting all variables into a single block, we can regard LALM as a special case of Algorithm \ref{alg:rpdc} with $m=1$, and YALL1 is a linearized ADMM with penalty parameter adaptively updated based on primal and dual residuals. 

The matrix $\vA\in\RR^{q\times 1000}$ in \eqref{eq:bp} was randomly generated with $q$ varying among $\{200, 300, 400\}$, and its entries independently follow standard Gaussian distribution. We normalized each row of $\vA$. A sparse vector $\vx^o$ was then generated with 30 nonzero entries that follow standard Gaussian distribution and whose locations are chosen uniformly at random. The vector $\vb=\vA\vx^o$. We evenly partitioned the variable $\vx$ into 100 blocks, and we set $\rho=\frac{\beta}{100}$ and $\vP_i=\beta\|\vA_i\|^2\vI,\, i=1,\ldots, 100$, where $\|\vA_i\|$ denotes the spectral norm of $\vA_i$. For LALM, we treated it as a special case of Algorithm \ref{alg:rpdc} with a single block and set $\rho=\beta$ and $\vP=\beta\|\vA\|^2\vI$. The same values of $\beta$ were used for both Algorithm \ref{alg:rpdc} and LALM. The parameters of YALL1 were set to the default values.

\begin{figure}
\begin{center}
\begin{tabular}{ccc}
$\beta=1$ & $\beta=10$ & $\beta=100$\\[0.1cm]
\includegraphics[width=0.3\textwidth]{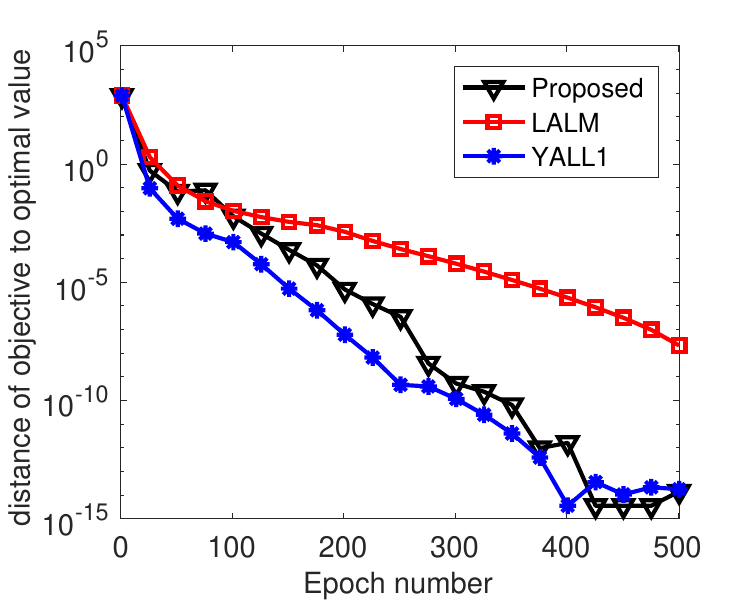}&
\includegraphics[width=0.3\textwidth]{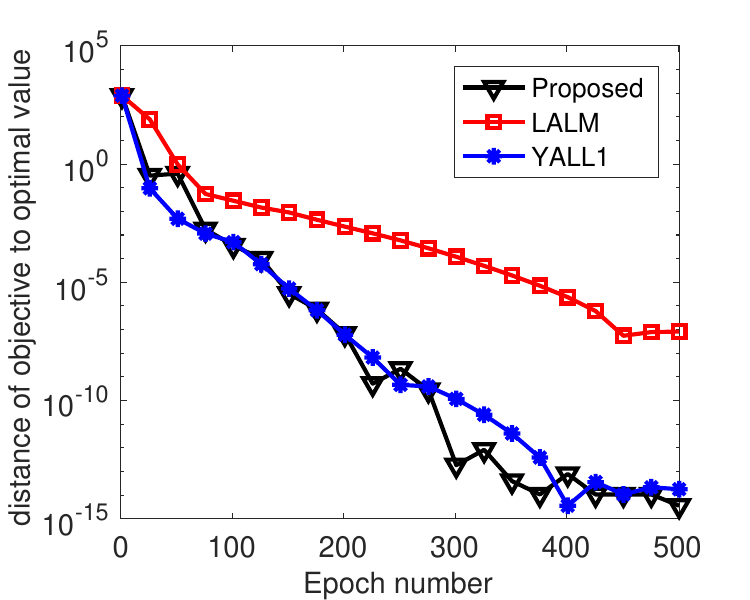}&
\includegraphics[width=0.3\textwidth]{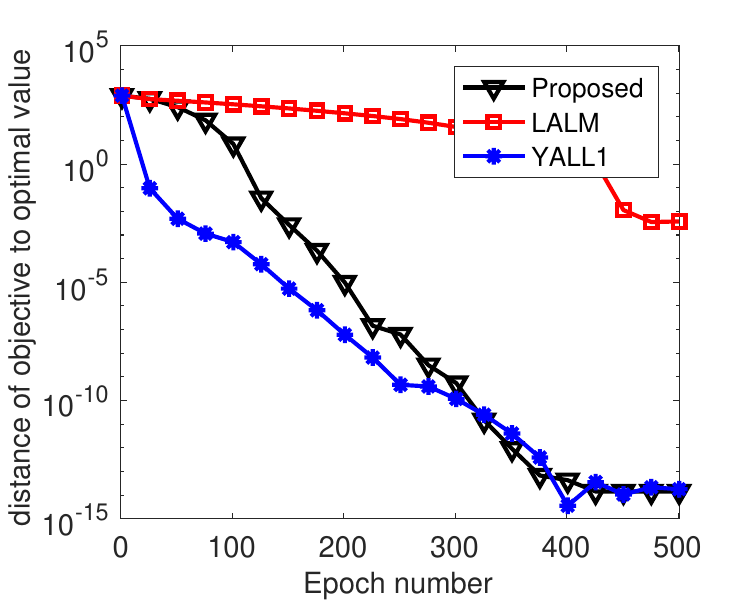}\\
\includegraphics[width=0.3\textwidth]{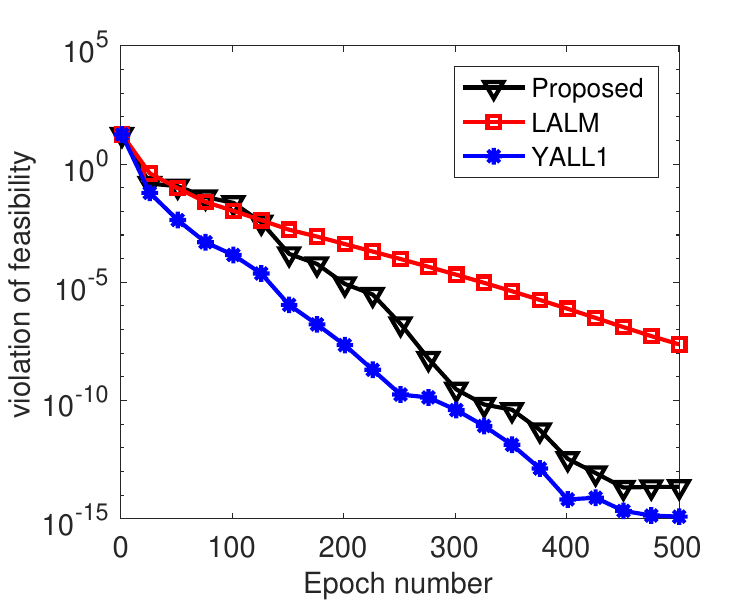}&
\includegraphics[width=0.3\textwidth]{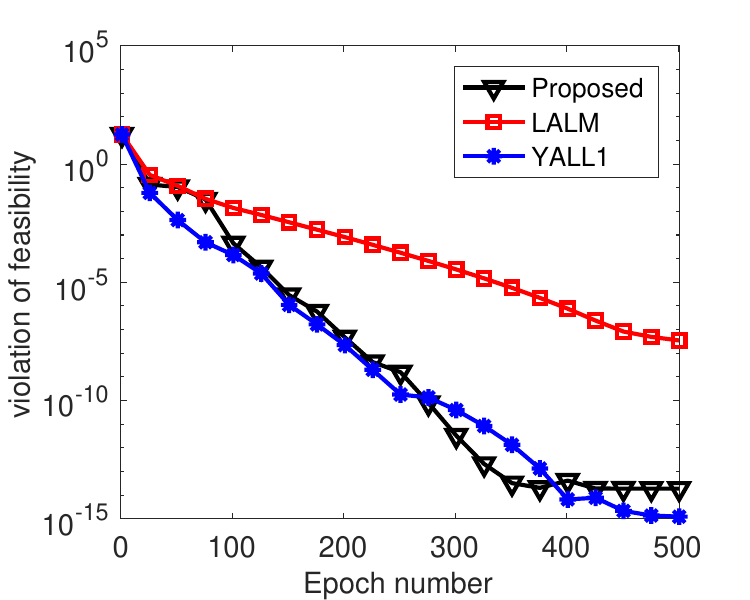}&
\includegraphics[width=0.3\textwidth]{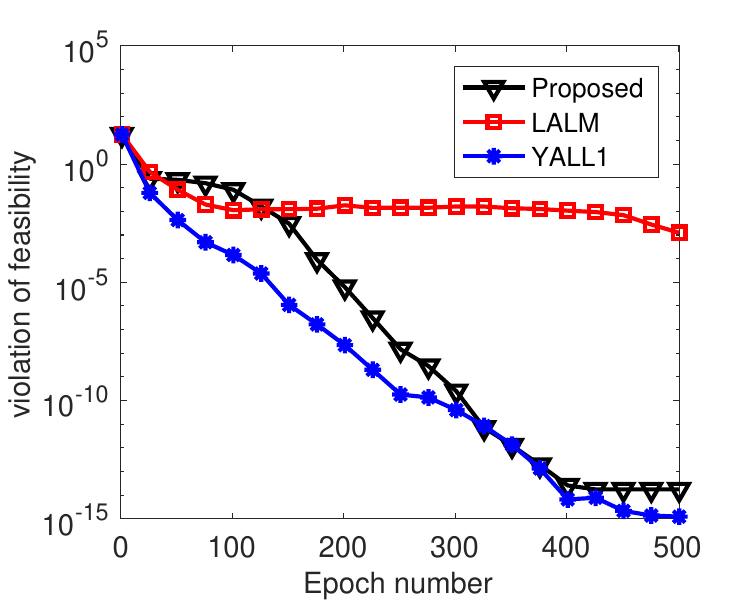}\\
\end{tabular}
\end{center}\caption{Results by three different algorithms on solving the basis pursuit problem \eqref{eq:bp} with $\vA\in\RR^{300\times 1000}$. The parameter $\beta$ varies among $\{1,10,100\}$ for Algorithm \ref{alg:rpdc} and LALM.}\label{fig:bp-beta}
\end{figure}

To compare the performance of the three algorithms, we plot their values of $|F(\vx^t)-F(\vx^*)|$ and $\|\vA\vx^t-\vb\|$ with respect to $t$, where $t$ denotes the epoch number.\footnote{Each epoch is equivalent to updating $m$ $\vx$-blocks.} Since the three algorithms have roughly the same per-epoch complexity, the plot in terms of running time will be similar. In Figure \ref{fig:bp-beta}, we fixed $q=300$ and varied $\beta$ among $\{1,10,100\}$. From the results, we see that the proposed algorithm perform significantly better than LALM and comparably as well as YALL1. In addition, the parameter $\beta$ affected both Algorithm \ref{alg:rpdc} and LALM but the former was only weakly affected. In Figure \ref{fig:bp-beta-m}, we set $\beta=\sqrt{q}$ and varied $q$ among $\{200,300,400\}$. Again we see that the proposed algorithm is significantly better than LALM. For $q=200$, Algorithm \ref{alg:rpdc} is slightly better than YALL1, and for $q=300$ and $400$, they perform equally well.

\begin{figure}
\begin{center}
\begin{tabular}{ccc}
$q=200$ & $q=300$ & $q=400$\\[0.1cm]
\includegraphics[width=0.3\textwidth]{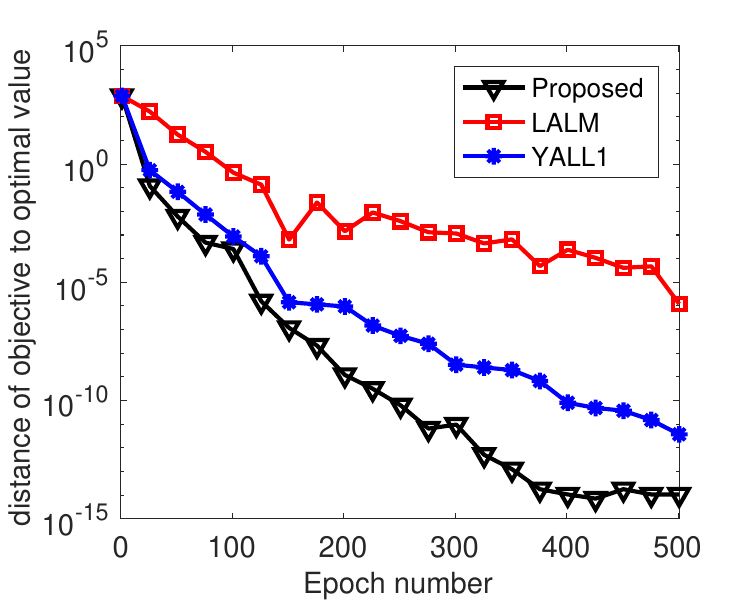}&
\includegraphics[width=0.3\textwidth]{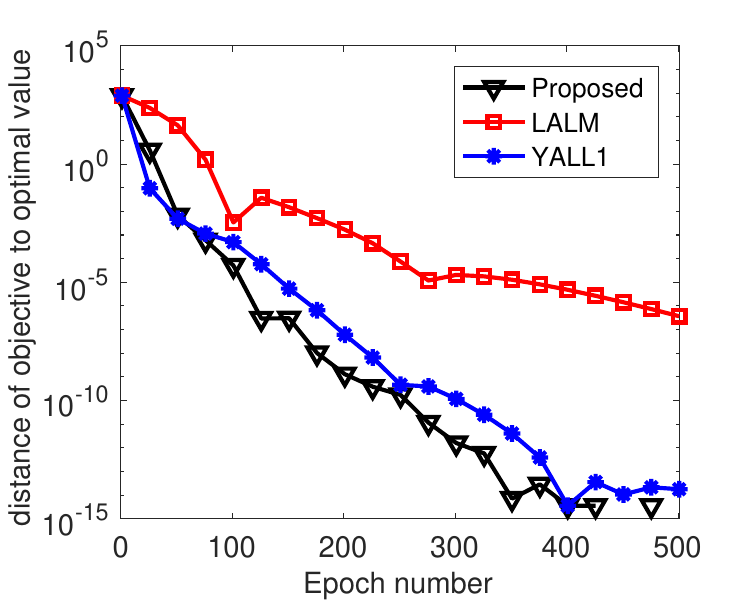}&
\includegraphics[width=0.3\textwidth]{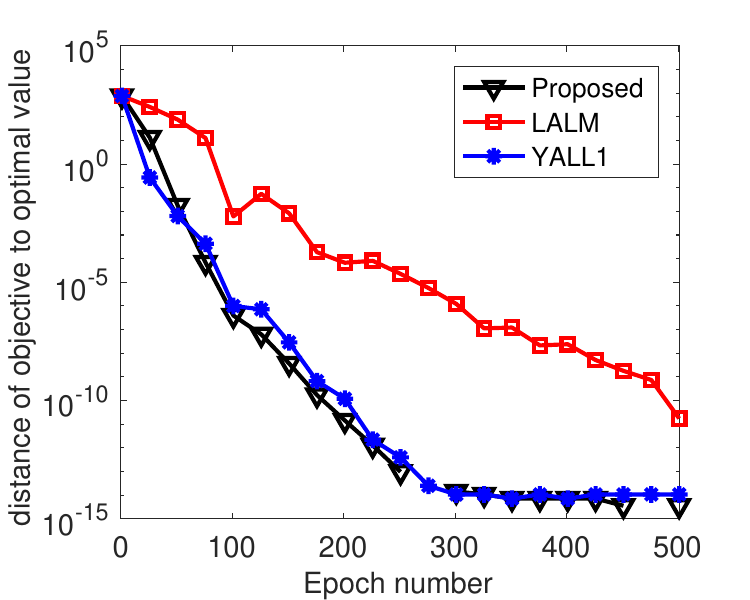}\\
\includegraphics[width=0.3\textwidth]{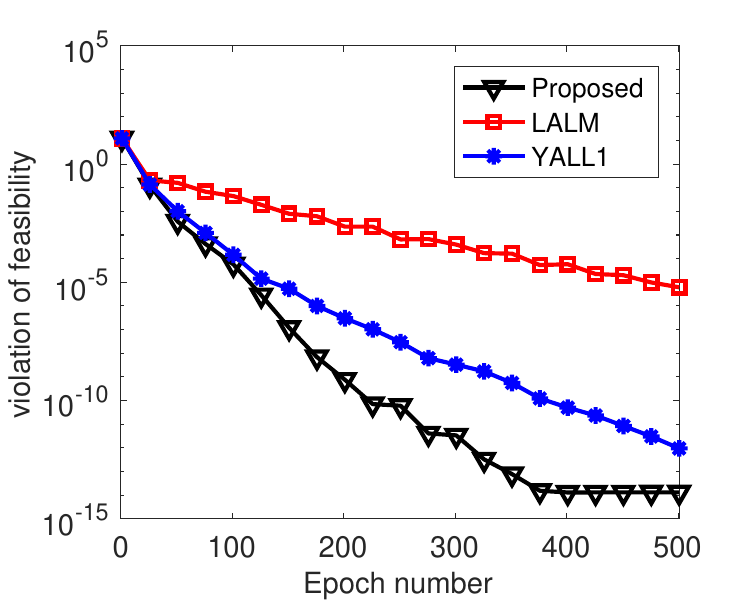}&
\includegraphics[width=0.3\textwidth]{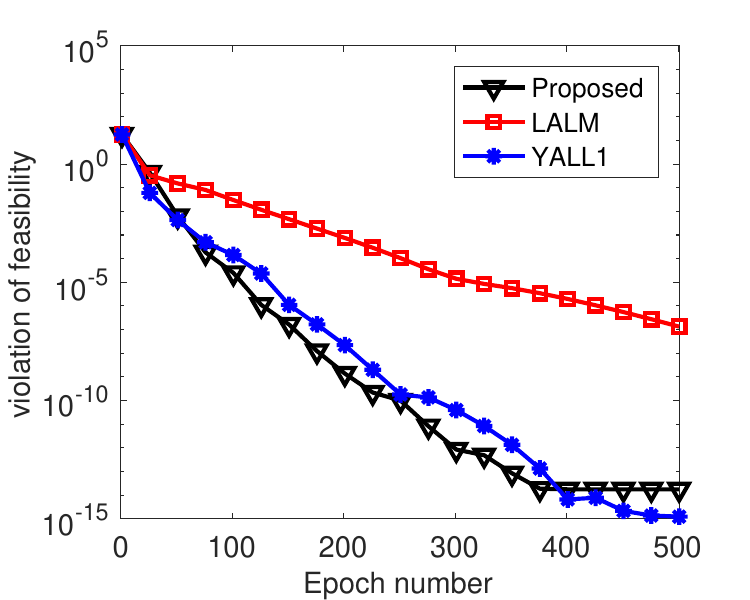}&
\includegraphics[width=0.3\textwidth]{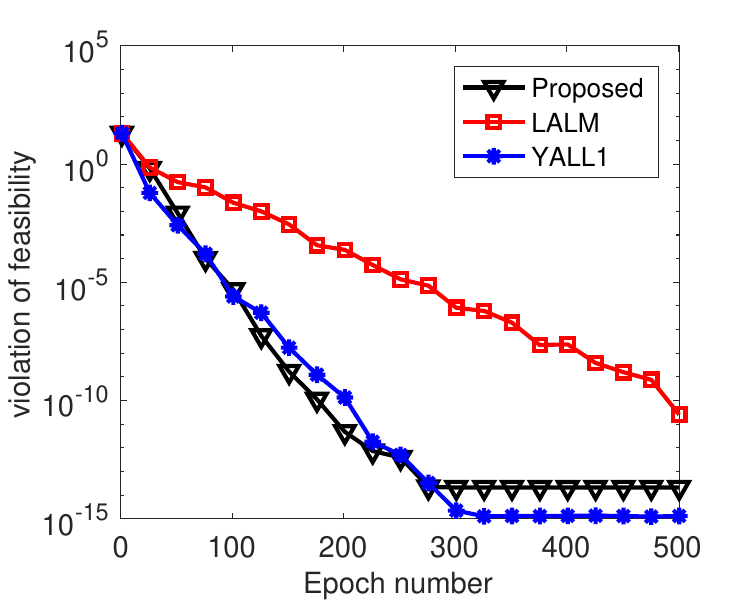}\\
\end{tabular}
\end{center}\caption{Results by three different algorithms on solving the basis pursuit problem \eqref{eq:bp} with $\vA\in\RR^{q\times 1000}$ and $q$ varying among $\{200,300,400\}$. The parameter $\beta$ was set to $\sqrt{q}$ for Algorithm \ref{alg:rpdc} and LALM.}\label{fig:bp-beta-m}
\end{figure}

\subsection{Quadratic programming}\label{sec:qp}

In this subsection, we simulate the performance of Algorithm \ref{alg:async-rpdc} with different delays on solving the nonnegativity constrained quadratic programming (NCQP): 
\begin{equation}\label{eq:ncqp}
\min_\vx \frac{1}{2}\vx^\top \vQ \vx + \vc^\top \vx, \st \vA\vx =\vb, \vx\ge 0,
\end{equation}
where $\vQ$ is a positive semidefinite matrix. We set $\vQ=\vH\vH^\top$ with $\vH\in\RR^{2000\times 2000}$ randomly generated from standard Gaussian distribution, and the vector $\vc$ was generated from Gaussian distribution. The matrix $\vA=[\vB, \vI]\in\RR^{200\times 2000}$ with the entries of $\vB$ independently following standard Gaussian distribution, and $\vb$ was generated from uniform distribution on $[0,1]$. This way, we guarantee the feasibility of \eqref{eq:ncqp}.  

\begin{figure}
\begin{center}
\begin{tabular}{cc}
\includegraphics[width=0.35\textwidth]{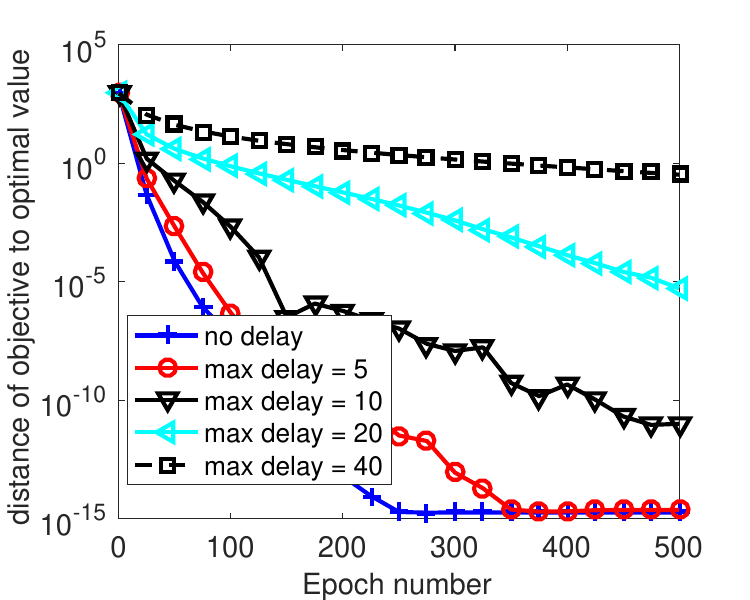}&
\includegraphics[width=0.35\textwidth]{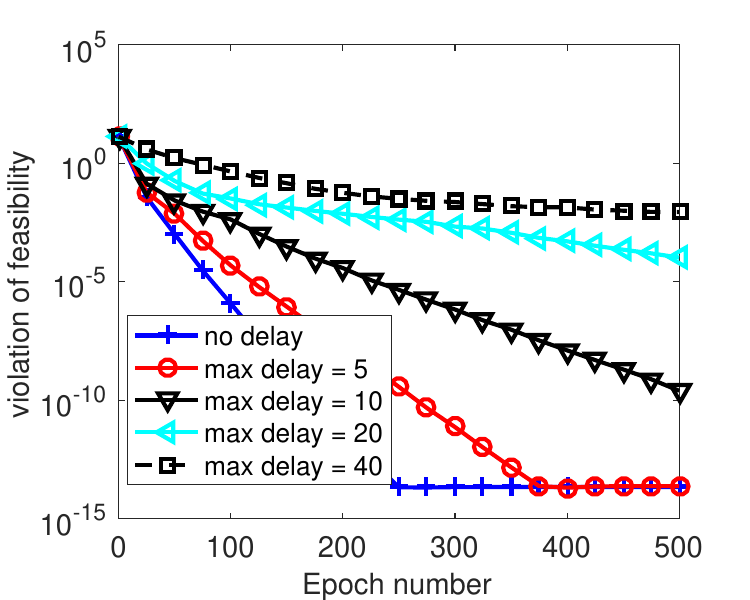}
\end{tabular}
\end{center}\caption{Results by Algorithm \ref{alg:async-rpdc} on solving the quadratic programming \eqref{eq:ncqp}. The matrices $\vP_i$'s are set according to \eqref{eq:async-P} with $\alpha=1$.}\label{fig:qp-nidstep}
\end{figure}

We partitioned $\vx$ into 2,000 blocks, namely, every coordinate was treated as one block. To see how the algorithm is affected by delayed block gradients, $\tau+1$ most recent iterates were kept, and $\hat{\vx}^k$ was set to one of these iterates that was chosen uniformly at random. We varied $\tau$ among $\{0, 5, 10, 20, 40\}$. $\beta$ was tuned to $\sqrt{2}$, $\rho=\frac{\beta}{2000}$ was used, and $\vP_i$'s were set in two different ways. Figure \ref{fig:qp-nidstep} plots the results by Algorithm \ref{alg:async-rpdc} with $\vP_i$'s set according to \eqref{eq:async-P} with $\alpha=1$. Note that for this instance, we have $L_i=Q_{ii}$, i.e., the $i$-th diagonal entry of $\vQ$ for each $i$, and $L_r=\max_i\|\vq_i\|$ where $\vq_i$ denotes the $i$-th column of $\vQ$. From the figure, we see that the convergence speed of the algorithm is affected by the delays. Larger $\tau$ gives smaller stepsize and leads to slower convergence. However, the algorithm is hardly affected by delayed block gradient if the same $\vP_i$'s were used, as shown in Figure \ref{fig:qp-idstep}. 
Practically, the maximum delay $\tau$ is unknown, but the results in Figure \ref{fig:qp-idstep} indicate that we can simply set $\vP_i$'s according to Theorem \ref{thm:cvg-prob} regardless of the delay. This implies that our analysis may not be tight. 

\begin{figure}
\begin{center}
\begin{tabular}{cc}
\includegraphics[width=0.35\textwidth]{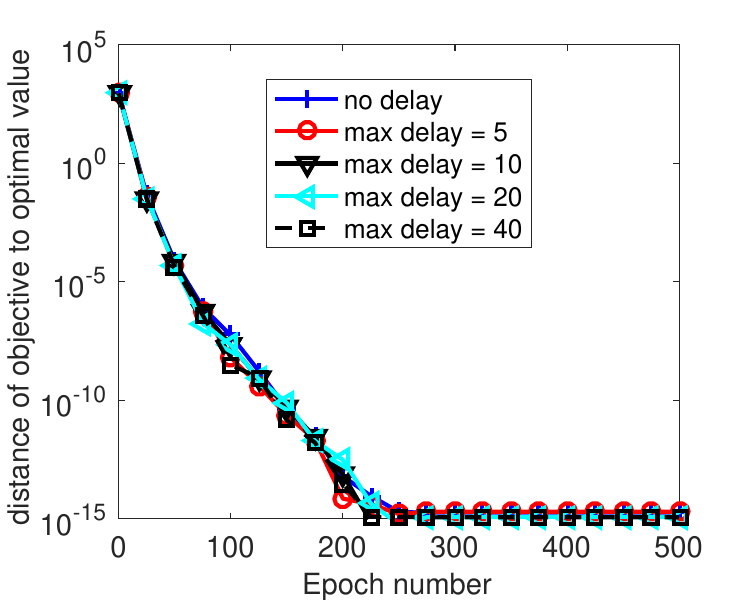}&
\includegraphics[width=0.35\textwidth]{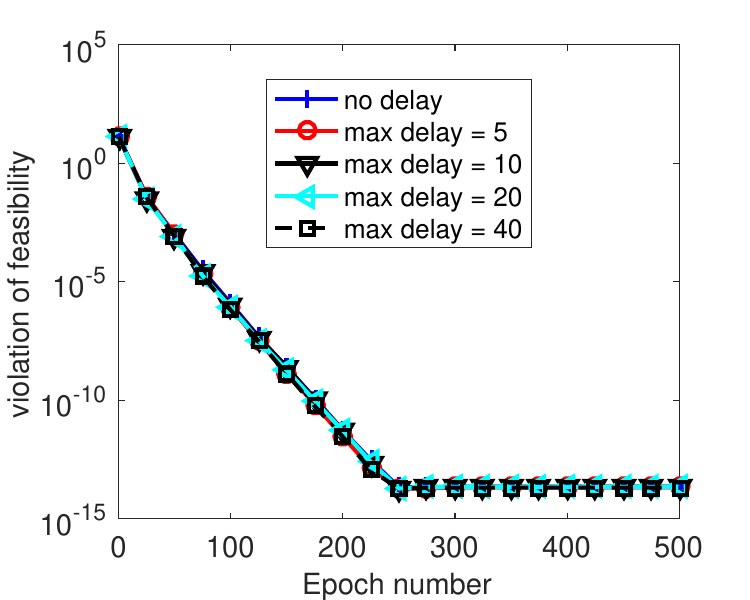}
\end{tabular}
\end{center}\caption{Results by Algorithm \ref{alg:async-rpdc} on solving the quadratic programming \eqref{eq:ncqp}. The same matrices $\vP_i$'s are used for different delays, i.e., $\vP_i= Q_{ii} + \beta \|\va_i\|^2,\, \forall i$ according to Theorem \ref{thm:cvg-prob}, where $\va_i$ is the $i$-th column of $\vA$.}\label{fig:qp-idstep}
\end{figure}

\subsection{Support vector machine}\label{sec:svm}
In this subsection, we compare the performance of the async-parallel Algorithm \ref{alg:async-rpdc} and its sync-parallel counterpart on solving the dual SVM \eqref{eq:dsvm}. Another way of parallel computing on solving \eqref{eq:dsvm} is to directly distribute computation of an algorithm (that may not be BCU type) over multiple nodes, such as the method in \cite{zeng2008fast}. In the test, we used two LIBSVM datasets:\footnote{The data can be downloaded from \url{https://www.csie.ntu.edu.tw/~cjlin/libsvmtools/datasets/}} \verb|rcv1| and \verb|news20|, whose characteristics are listed in Table \ref{table:svm-data}. 

\begin{table}
\begin{center}
{\small
\begin{tabular}{cccc}
\hline
Name & \#samples & \#features & \#nonzeros\\\hline
rcv1 & 20,242 & 47,236 & 1,498,952\\
news20 & 19,996 & 1,355,191 & 9,097,916\\\hline
\end{tabular}
}
\end{center}\caption{Characteristics of two LIBSVM datasets}\label{table:svm-data}
\end{table}

We partitioned the variable into blocks of size 50 or 51. For both sync and async-parallel methods, $\beta=0.1$ and $\rho=\frac{\beta}{m}$ were set, where $m$ is the number of blocks. As suggested in section \ref{sec:qp}, for the async-parallel method, we set $\vP_i= (L_i +\beta\|\vA_i\|^2)\vI,\,\forall i$ according to Theorem \ref{thm:cvg-prob}. For the sync-parallel method, if there are $p$ cores, we selected a set $S$ of $p$ blocks at every iteration and set $\vP_i = \sum_{j\in S}(L_j+\beta\|\vA_j\|^2)\vI$ for all $i\in S$. We also used $\vP_i$'s the same as those by the async-parallel method but noticed that the sync-parallel method diverged. The larger weight matrices are also suggested in \cite{GXZ-RPDCU2016} to be proportional to the number of blocks. Note that in the dual SVM \eqref{eq:dsvm}, if we let $\vX_i$ and $\vy_i$ contain the data points and labels corresponding to the $i$-th block variable, then $L_i$ equals the spectral norm of the matrix $\diag(\vy_i)\vX_i^\top \vX_i \diag(\vy_i)$. Since every block only has 50 or 51 coordinates, it is easy to compute $L_i$'s.

We ran the tests on a machine with 20 cores. Figure \ref{fig:svm-rcv1} plots the results by the sync and async-parallel algorithms on the \verb|rcv1| dataset. From the figure, we see that in terms of epoch number, the sync-parallel method converges slower if more cores are used, while the async-parallel one converges almost the same with different number of cores. As shown in Figure \ref{fig:svm-news20}, similar results were observed for the \verb|news20| dataset. We also measured the speed-up of the two parallel methods in terms of running time. The results are plotted in Figure \ref{fig:svm-sp}. From the results, we see that the async-parallel method achieves significantly better speed-up than the sync-parallel one, and that is because synchronization at every iteration wastes much waiting time.

\begin{figure}
\begin{center}
\begin{tabular}{cc}
\includegraphics[width=0.3\textwidth]{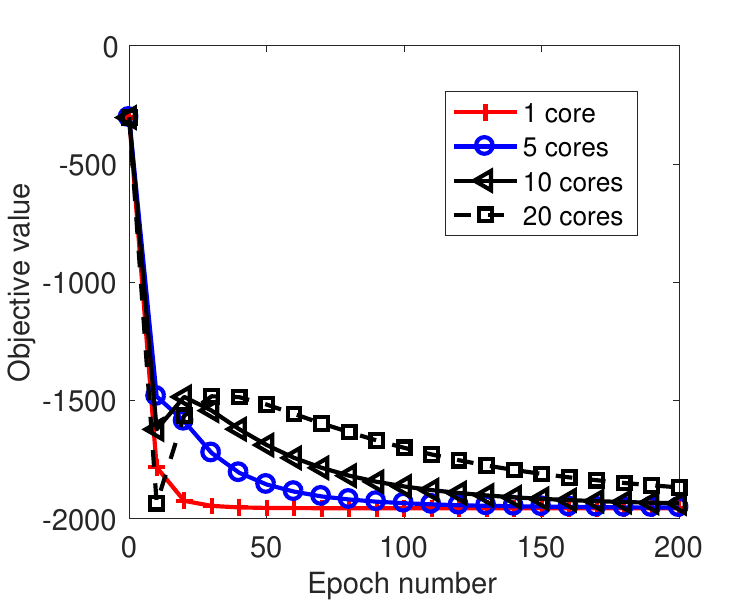}&
\includegraphics[width=0.3\textwidth]{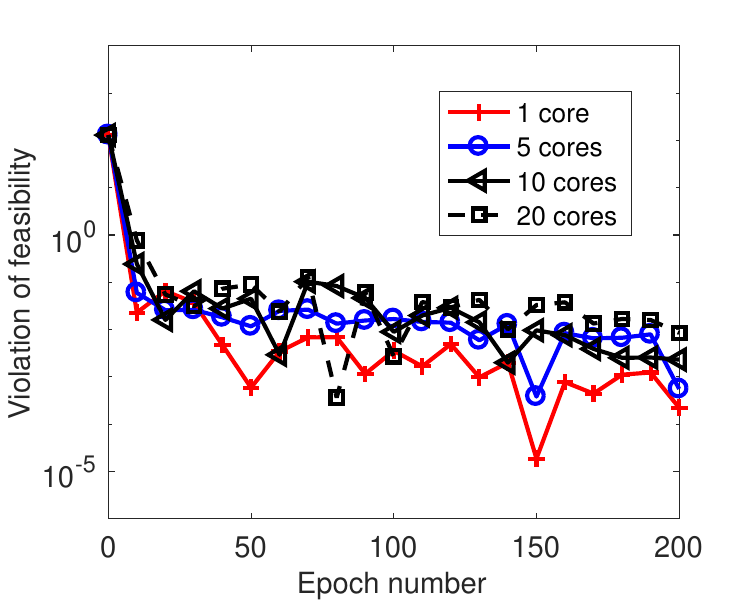}\\
\includegraphics[width=0.3\textwidth]{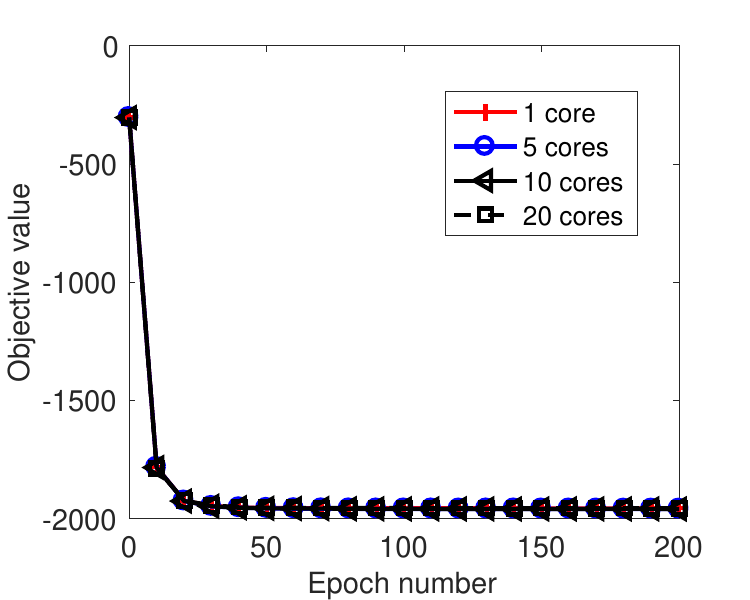}&
\includegraphics[width=0.3\textwidth]{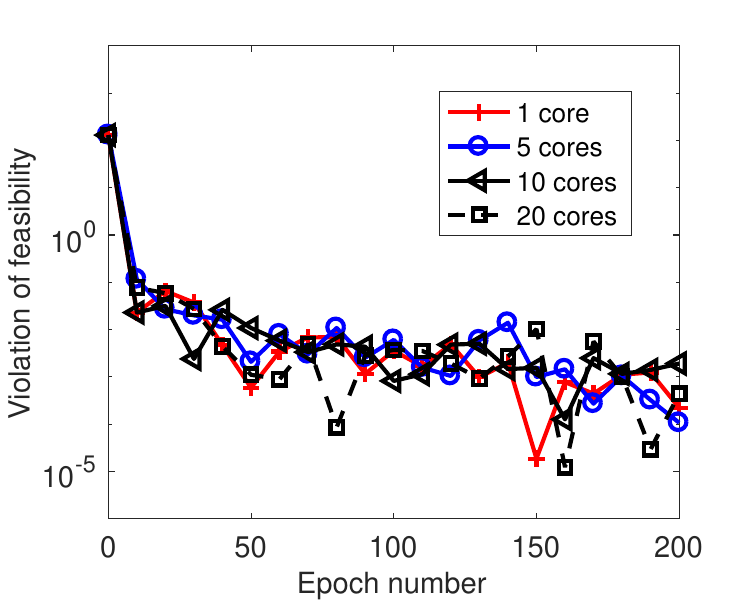}
\end{tabular}
\end{center}\caption{Results by the sync-parallel (top) and async-parallel (bottom) algorithms on solving the dual SVM \eqref{eq:dsvm}. The dataset rcv1 is used.}\label{fig:svm-rcv1}
\end{figure}

\begin{figure}
\begin{center}
\begin{tabular}{cc}
\includegraphics[width=0.3\textwidth]{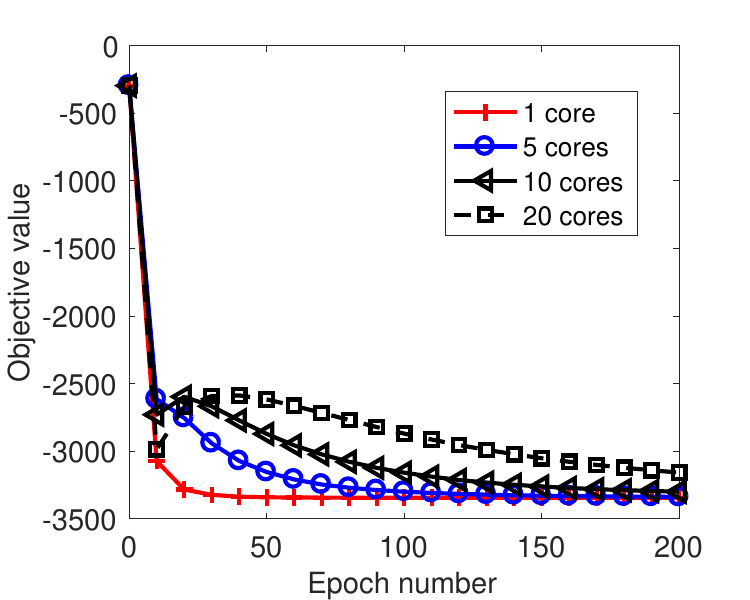}&
\includegraphics[width=0.3\textwidth]{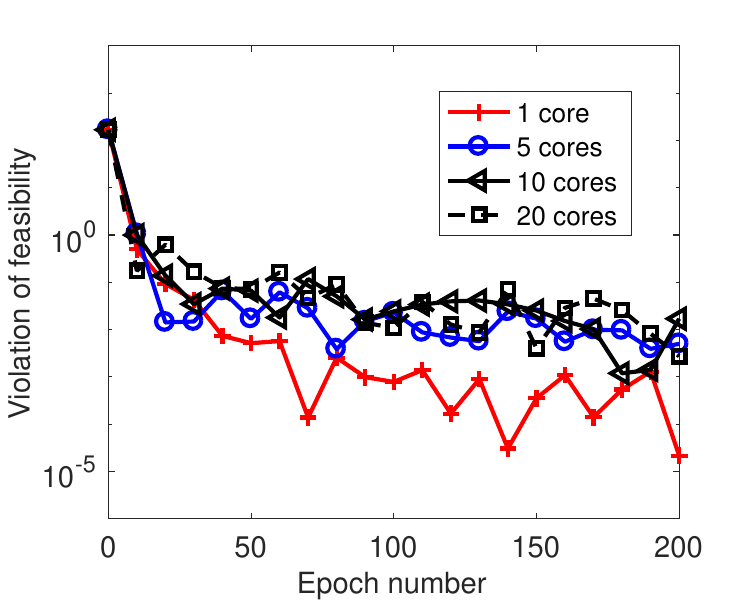}\\
\includegraphics[width=0.3\textwidth]{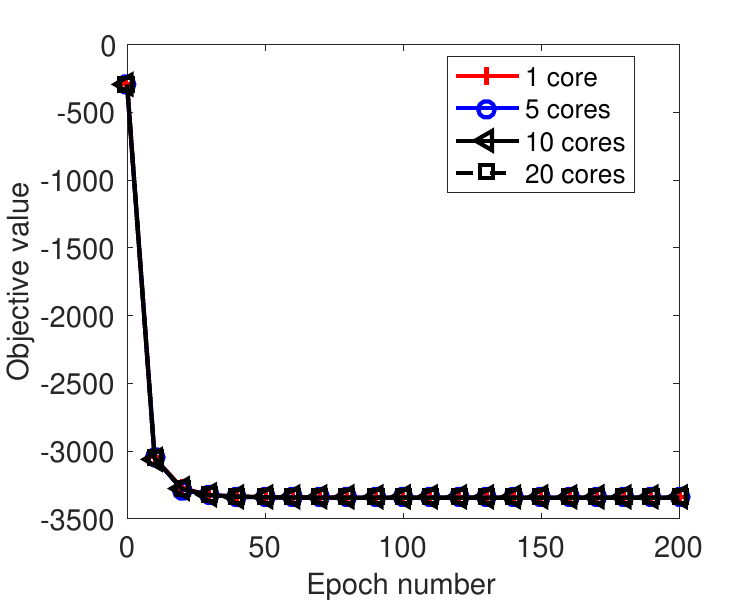}&
\includegraphics[width=0.3\textwidth]{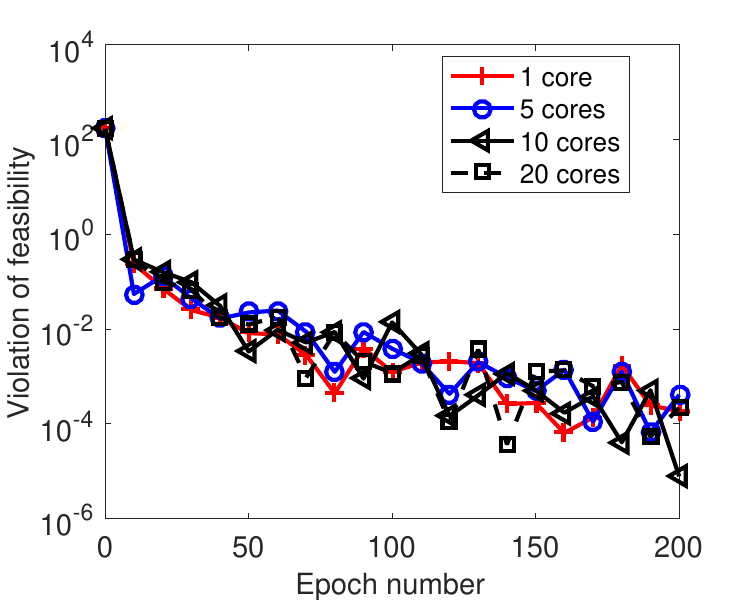}
\end{tabular}
\end{center}\caption{Results by the sync-parallel (top) and async-parallel (bottom) algorithms on solving the dual SVM \eqref{eq:dsvm}. The dataset news20 is used.}\label{fig:svm-news20}
\end{figure}

\begin{figure}
\begin{center}
\begin{tabular}{cc}
rcv1 & news20\\
\includegraphics[width=0.3\textwidth]{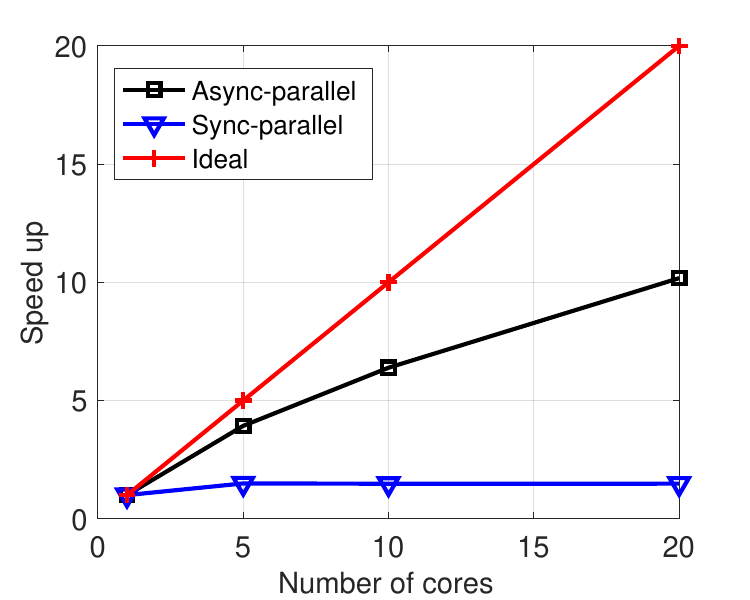}&
\includegraphics[width=0.3\textwidth]{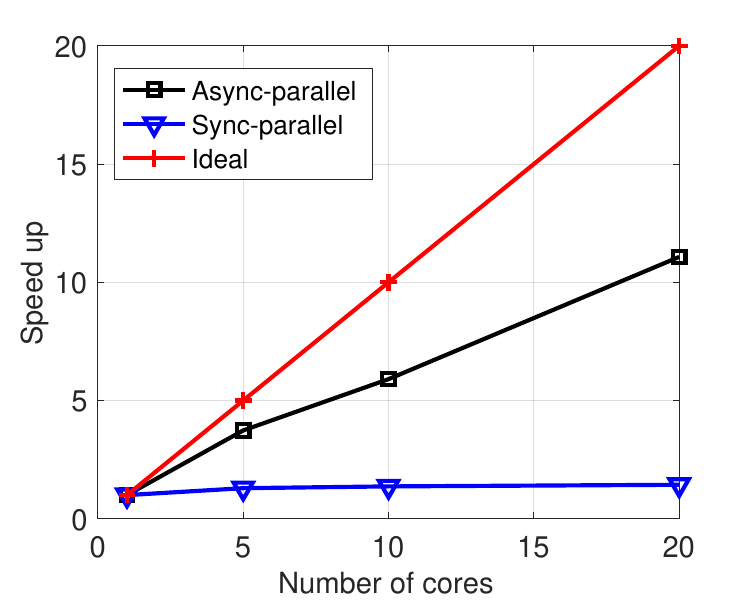}
\end{tabular}
\end{center}\caption{Speed up of the sync and async-parallel algorithms for solving the dual SVM \eqref{eq:dsvm} on different number of cores.}\label{fig:svm-sp}
\end{figure}

\section{Conclusions}\label{sec:conclusion}
We have proposed an async-parallel primal-dual BCU method for convex programming with \emph{nonseparable} objective and arbitrary linear constraint. As a special case on a single node, the method reduces to a randomized primal-dual BCU for multi-block linearly constrained problems. Convergence and also rate results in probability have been established under convexity assumption. We have also numerically compared the proposed algorithm to several existing methods. The experimental results demonstrate the superior performance of our algorithm over other ones.  

\bibliographystyle{abbrv}

\end{document}

%% file: macros.tex
\usepackage{mathtools}

\usepackage[normalem]{ulem} 

\newcommand{\bm}[1]{\boldsymbol{#1}}

\newcommand{\va}{{\mathbf{a}}}
\newcommand{\vb}{{\mathbf{b}}}
\newcommand{\vc}{{\mathbf{c}}}

\newcommand{\ve}{{\mathbf{e}}}

\newcommand{\vq}{{\mathbf{q}}}
\newcommand{\vr}{{\mathbf{r}}}

\newcommand{\vv}{{\mathbf{v}}}

\newcommand{\vx}{{\mathbf{x}}}
\newcommand{\vy}{{\mathbf{y}}}

\newcommand{\vA}{{\mathbf{A}}}
\newcommand{\vB}{{\mathbf{B}}}

\newcommand{\vH}{{\mathbf{H}}}
\newcommand{\vI}{{\mathbf{I}}}

\newcommand{\vL}{{\mathbf{L}}}

\newcommand{\vP}{{\mathbf{P}}}
\newcommand{\vQ}{{\mathbf{Q}}}

\newcommand{\vU}{{\mathbf{U}}}

\newcommand{\vX}{{\mathbf{X}}}

\newcommand{\vlam}{{\bm{\lambda}}}
\newcommand{\vtheta}{{\bm{\theta}}}

\newcommand{\cL}{{\mathcal{L}}}

\newcommand{\EE}{\mathbb{E}} 
\newcommand{\RR}{\mathbb{R}} 
\newcommand{\vzero}{\mathbf{0}} 

\newcommand{\Prob}{{\mathrm{Prob}}} 
\newcommand{\prox}{{\mathbf{prox}}} 
\newcommand{\Diag}{{\mathrm{Diag}}} 

\newcommand{\st}{\mbox{ s.t. }}

\DeclareMathOperator*{\argmin}{arg\,min} 

\newcommand{\bc}{\begin{center}}
\newcommand{\ec}{\end{center}}

\newcommand{\bdm}{\begin{displaymath}}
\newcommand{\edm}{\end{displaymath}}

\newcommand{\beq}{\begin{equation}}
\newcommand{\eeq}{\end{equation}}

\newcommand{\bfl}{\begin{flushleft}}
\newcommand{\efl}{\end{flushleft}}

\newcommand{\bt}{\begin{tabbing}}
\newcommand{\et}{\end{tabbing}}

\newcommand{\beqn}{\begin{eqnarray}}
\newcommand{\eeqn}{\end{eqnarray}}

\newcommand{\beqs}{\begin{align*}} 
\newcommand{\eeqs}{\end{align*}}  

\newtheorem{remark}{Remark}[section]
\newtheorem{assumption}{Assumption}